\documentclass[9pt]{amsart}

\usepackage{soul}
\usepackage{amsmath,graphicx}
\usepackage{amssymb}
\usepackage{amsfonts, amsthm}
\usepackage{todonotes}
\usepackage{enumitem, dsfont}
\usepackage{aliascnt}

\usepackage[parfill]{parskip}

\usepackage{color}
\usepackage{xcolor}

\usepackage[colorlinks=true, pdfstartview=FitV,linkcolor=blue,citecolor=blue, urlcolor=black]{hyperref}

\makeatletter
\def\subsection{\@startsection{subsection}{2}%
  \z@{.5\linespacing\@plus.7\linespacing}
{.5\baselineskip}%
  {\normalfont\centering\scshape}%
}
\makeatother

\makeatletter
\def\l@subsection{\@tocline{2}{0pt}{2.5pc}{5pc}{}}
\def\l@section{\@tocline{1}{0pt}{0pc}{5pc}{\bfseries}}
\makeatother

\numberwithin{equation}{section}

\def\R{\mathbb{R}}

\def\N{\mathbb{N}}

\def\r3{\R^3}

\newcommand{\ben}{\begin{eqnarray*}}
\newcommand{\een}{\end{eqnarray*}}

\newtheorem{Defi}{Definition}

\newtheorem{theorem}{Theorem}

\newaliascnt{proposition}{theorem}
\newtheorem{proposition}[proposition]{Proposition}
\aliascntresetthe{proposition}

\newtheorem*{proposition*}{Proposition}

\newaliascnt{corollary}{theorem}
\newtheorem{corollary}[corollary]{Corollary}
\aliascntresetthe{corollary}

\newaliascnt{lemma}{theorem}
\newtheorem{lemma}[lemma]{Lemma}
\aliascntresetthe{lemma}

\newtheorem{remark}{Remark}
\newtheorem*{remark*}{Remark}
\newtheorem*{claim*}{Claim}

\numberwithin{lemma}{section}
\numberwithin{theorem}{section}
\numberwithin{corollary}{section}
\numberwithin{proposition}{section}

\usepackage{cleveref}
\crefformat{equation}{(#2#1#3)}

\crefname{theorem}{Theorem}{Theorems}
\Crefname{theorem}{Theorem}{Theorems}

\crefname{lemma}{Lemma}{Lemmas}
\Crefname{lemma}{Lemma}{Lemmas}

\crefname{proposition}{Proposition}{Propositions}
\Crefname{proposition}{Proposition}{Propositions}

\crefname{corollary}{Corollary}{Corollaries}
\Crefname{corollary}{Corollary}{Corollaries}

\crefname{Defi}{Definition}{Definitions}
\Crefname{Defi}{Definition}{Definitions}

\crefname{remark}{Remark}{Remarks}
\Crefname{remark}{Remark}{Remarks}

\crefname{proposition*}{Proposition}{Propositions}
\Crefname{proposition*}{Proposition}{Propositions}
\crefname{remark*}{Remark}{Remarks}
\Crefname{remark*}{Remark}{Remarks}
\crefname{claim*}{Claim}{Claims}
\Crefname{claim*}{Claim}{Claims}

\def\mk{\mathrm{d}}
\DeclareMathOperator{\supp}{supp}
\DeclareMathOperator{\Leb}{\mathcal{L}}
\let\div\relax
\DeclareMathOperator{\div}{div}
\def\normal{\vec{n}}
\newcommand{\inner}[2]{\left\langle #1, #2\right\rangle}
\DeclareMathOperator{\Id}{Id}

\begin{document}

\title[Stability of optimal transport and second variation of the distance]{
  Stability of optimal transport maps and second variation of the $2$-Monge--Kantorovich distance
  }

\author{F.-U. Caja-Lopez}

\address{Department of Mathematics, The University of Texas at Austin (USA)}
\email{funai.caja@utexas.edu}

\author{Matias G. Delgadino}
\address{Department of Mathematics, The University of Texas at Austin (USA)}
\email{matias.delgadino@math.utexas.edu}

\author{Jun Kitagawa}
\address{Department of Mathematics, Michigan State University (USA)}
\email{kitagawa@math.msu.edu}


\begin{abstract}
We establish several quantitative stability estimates for optimal transport maps between non-degenerate densities on uniformly convex domains for the quadratic cost. Under H\"older regularity assumptions, we prove Lipschitz $L^{2}$ (respectively $\mathcal C^{0,\alpha}$) stability estimates for optimal transport maps in terms of the $2$-Monge--Kantorovich distance (respectively $L^{p}$ distances) between pairs of source and target densities. When the continuity assumption is removed, we obtain a Lipschitz $L^{2}$ stability estimate for the Brenier potentials in terms of the $L^{2}$ distance between the source and target densities.

The proofs rely on a precise characterization of the linear response of the Brenier potential along smooth interpolations of the data, obtained by linearizing the Monge--Amp\`ere equation in divergence form. As a further application of this approach, we derive an explicit formula for the second variation of the quadratic Monge--Kantorovich distance.
\end{abstract}

\keywords{Optimal transport, quantitative stability}

\subjclass[2020]{
35J96, 
49K40, 
49Q22} 
\maketitle

\setcounter{tocdepth}{2}
\begingroup
\hypersetup{linkcolor=black}
\endgroup


\section{Introduction}

Given $1\leq p<\infty$, two probability measures $\mu$, $\nu\in \mathcal{P}_p(\mathbb R^d)$ (the space of Borel probability measures on $\R^d$ with finite $p$th order moments), the Monge problem in optimal transport for the $p$th power cost $c(x,y)=\lvert x-y\rvert^p$ consists in minimizing
\begin{equation}\label{eq:Monge}
    \int_{\mathbb R^d}\lvert T(x)-x\rvert^p\,d\mu(x)
\end{equation}
among all maps $T:\mathbb R^d\longrightarrow\mathbb R^d$ satisfying $T_\sharp\mu = \nu$, where $\sharp$ denotes the push forward of a measure. The $p$th root of the minimal value in~\eqref{eq:Monge} above is denoted by $\mk_p(\mu,\nu)$, which is known as the \emph{$p$-Monge--Kantorovich} distance\footnote{This is also sometimes called the \emph{$p$-Wasserstein distance.}}. In \cite{brenier1987decomposition} Brenier showed that when $p=2$ and the source measure is absolutely continuous, there exists a $\mu$-a.e. unique optimal map given by $T(x)=\nabla\phi(x)$, where $\phi$ is a convex function. Under further regularity assumptions on $\mu$, $\nu$ and convexity of their supports, various regularity results on $\phi$ were later proven in \cite{caffarelli1990interior_W2p_estimates, Delanoe91, caffarelli1992reg_convex_potential, caffarelli1992boundary_reg_convex_potential, caffarelli1996boundary_reg_convex_potential, Urbas97} by studying a Monge--Amp\`ere PDE of the form
$$
\det D^2\phi(x)=\frac{f(x)}{g(\nabla\phi(x))},\qquad\nabla\phi(\textnormal{supp}\,\mu)\subset \textnormal{supp}\,\nu,
$$
which is satisfied by $\phi$. These regularity results were later extended to other settings in works such as \cite{DelanoeGe10, DelanoeGe11, FigalliPhil2013W21RegMongeAmpere, FigalliKimMcCann13, GuillenKitagawa15, figalli2019regularity_MA_unbounded, chenLiuWang2021MAglobalReg} among many others. 

Regarding stability of solutions, it is not hard to prove \emph{qualitative} results using compactness arguments. For example, it is known that weak convergence of the target measure implies convergence in measure of the optimal map (see \cite[Corollary 5.23]{villani2009optimal_old_new} for a proof in a very general setting). Similarly in the case $p=2$, Caffarelli proved in \cite{caffarelli1992boundary_reg_convex_potential} a $\mathcal C^{1,\alpha}$ estimate on $\phi$ which implies uniform stability of maps due to the Arzelà--Ascoli theorem. However, \emph{quantitative} stability estimates have proven to be significantly more difficult. This is a currently active research area and one of the main objectives of this manuscript. 

Let us give a brief literature review of the current state of quantitative stability estimates in optimal transport for minimizers of~\eqref{eq:Monge} when $p=2$. To do so, it is convenient to introduce some terminology.
\begin{Defi}\label{def: kantorovich}
    Let $\mu$, $\nu\in \mathcal{P}_2(\R^d)$. We call a minimizer of~\eqref{eq:Monge} with $p=2$ an \emph{optimal map from $\mu$ to $\nu$}. An \emph{optimal plan from $\mu$ to $\nu$} is a minimizer of the functional
    \begin{align*}
        \pi\mapsto \int_{\R^d\times \R^d} \lvert x-y\rvert^2 d\pi(x, y)
    \end{align*}
    over all $\pi\in \mathcal{P}(\R^d\times \R^d)$ whose left and right marginals are $\mu$ and $\nu$ respectively.

    We call any convex function $\phi$ such that $\nabla \phi$ is an optimal map from $\mu$ to $\nu$, a \emph{Brenier potential from $\mu$ to $\nu$}.
\end{Defi}

As mentioned above, when $\mu$ is absolutely continuous there always exists a Brenier potential from $\mu$ to $\nu$ by \cite[Theorem 1.3]{Brenier91}; additionally, the optimal map is $\mu$-a.e. uniquely determined, and under some conditions on the connectedness of $\supp \mu$ (satisfied if $\supp \mu$ is convex), the Brenier potential itself is unique up to translation by a constant. Additionally, by \cite[Proposition 4.2 and Theorem 4.4]{McCann97} if $\mu$, $\nu$ are absolutely continuous with densities $f$ and $g$ respectively, for $\mu$-a.e. $x$ it holds that
\begin{align}\label{eqn: MA eqn}
    g(\nabla \phi(x))\det D^2\phi(x)=f(x),
\end{align}
where $D^2\phi(x)$ is the Hessian in the sense of Aleksandrov (see \cite[(35)]{McCann97}). If $\phi$ is twice differentiable, the above holds everywhere with $D^2\phi$ the usual Hessian.

One of the first known quantitative stability results is due to Ambrosio and presented by Gigli in \cite{Gigli2011HolderContCurve}, which showed Hölder continuity with respect to the $2$-Monge--Kantorovich distance of a curve of optimal transport maps when they are sufficiently regular. We also mention the early work by Loeper in \cite{loeper2005regularity} that frames his results in terms of stability of Brenier's polar factorization applied to periodic maps. To state more recent results, suppose 
$$
f_0,f_1:\Omega \longrightarrow\mathbb R,\quad g_0,g_1:\Omega_*\longrightarrow\mathbb R,
$$
where $f_i$, $g_i$ are the density functions of absolutely continuous source and target measures in $\mathcal{P}_2(\R^d)$ respectively, and let $\phi_i$ be Brenier potentials between $f_i \Leb^d$ and  $g_i \Leb^d$ for  ${i=0,1}$. We will refer to the density function of an absolutely continuous measure as a \emph{probability density}, and by an abuse of notation, we will write the density to sometimes denote the measure itself, which will be clear from context.

Most results on the stability of optimal transport maps take the form
\begin{equation}\label{eq:stability_general}
    \Vert \nabla \phi_1-\nabla\phi_0\Vert_{L^2(\rho)} \leq C\Big( \mk_p(f_0,f_1)^\alpha +\mk_p(g_0,g_1)^\alpha\Big).
\end{equation}
Berman obtained in \cite{Berman2021StabilityOT} such an inequality with $g_0=g_1$ when one of the maps is Lipschitz, together with a more general version in which the exponent $\alpha$ deteriorates exponentially with the dimension. A series of subsequent works relaxed these assumptions and improved the exponent $\alpha$; see M{\'e}rigot, Delalande, and Chazal \cite{MerigotDelaChazal2020StabOT}, Delalande and M{\'e}rigot \cite{MerigotDela2023StabOt}, Letrouit and M{\'e}rigot \cite{LetroMeri2024gluingOT}, and M{\'e}rigot \cite{Merigot2026SharpStability}. In particular, when $p=1$ and $f_0=f_1=\rho$ has a bounded density, \cite{Merigot2026SharpStability} proves that \cref{eq:stability_general} holds with the sharp exponent $\alpha=\frac{1}{4}$. When $p=2$ and $f_0=f_1=\rho$ is the uniform measure on a convex body, the manuscript released by OpenAI \cite{OpenAI2026SharpOneThird} shows that \cref{eq:stability_general} holds with the sharp exponent $\alpha=\frac{1}{3}$.

These results have been extended in several directions. Ruan \cite{Ruan2026SharpConditioning} considers a source with density in $L^q$ and obtains \cref{eq:stability_general} with $p=1$ and $\alpha=\frac{q-1}{4q-2}$. Ford \cite{Ford2026QuantStabPlans} adapts the approach of \cite{Merigot2026SharpStability} to a source satisfying $\rho(B_\eta(x))\leq M\eta^{d-1+\theta}$, and obtains \cref{eq:stability_general} with $\alpha=\frac{\theta}{2(1+\theta)}$. Measures with unbounded support are treated by S.~Wang and X.~Wang \cite{WangWang2026StabRegularization}, who prove \cref{eq:stability_general} with $p=2$ and an exponent $\alpha$ determined by the integrability of the source. Since most of these results keep either the source or the target fixed, we also mention the work of Cazelles, Pauwels, and Portales \cite{CazellesPauwelsPortales2026}, which establishes \cref{eq:stability_general} with both measures varying, $p=1$, and $\alpha=\frac{1}{5}$.

Some of these results have also been extended beyond the Euclidean setting. On Riemannian manifolds, we mention the work of the third author with Letrouit and M{\'e}rigot \cite{KitaLetroMeri2025StabManifolds}, and that of Han and Zhu \cite{HanZhu2026SubRiemannian}, which also covers sub-Riemannian manifolds. Han and Zhu also studied in \cite{HanZhu2026MetricMeasure} metric measure spaces.

Under regularity assumptions on $\phi_0$, Balakrishnan and Manole showed in \cite{balakManole2025stabilityOT} that \cref{eq:stability_general} holds with $p=2$ and $\alpha=1$. Finally, we mention the work of Gonz{\'a}lez-Sanz, Mordant, and Sheng \cite{GonzalezSanzMordantSheng2026EmpiricalPotentials}, who obtain stability estimates for Brenier potentials in terms of a weak dual norm, using a technique closer the ones used in this paper.

Our first result can be seen as an improvement of the previous estimates, when the densities are non-degenerate and Hölder continuous.
\begin{theorem}\label{thm:d2_d2}
  Let $\Omega$, $\Omega_*\subset \mathbb{R}^d$ be bounded, uniformly convex sets with boundaries of class $\mathcal C^{3,\alpha}$ and $\mathcal{C}^{2, \alpha}$ respectively, for some $\alpha\in(0,1)$. Let $f_0$, $f_1\in\mathcal{C}^{0,\alpha}(\overline{\Omega})$ and $g_0$, $g_1\in\mathcal{C}^{0,\alpha}(\overline{\Omega}_*)$ be probability densities with $f_i$, $g_i\geq a>0$ on $\Omega$ and $\Omega_*$ respectively. Then, if $\phi_i$ and $\pi_i$ are Brenier potentials and optimal plans between $f_i$ and $g_i$ respectively,
  \begin{equation}\label{eq:L2_d2_bound}
    \left\Vert \nabla\phi_{1}-\nabla\phi_{0}\right\Vert _{L^{2}(\Omega)}
    \leq C\big(\mk_2(f_0,f_1)+\mk_2(g_0,g_1)\big),
  \end{equation}
  and
  \begin{equation}\label{eq:boundplans}
    \mk_2(\pi_0,\pi_1)\leq C\big(\mk_2(f_0,f_1)+\mk_2(g_0,g_1)\big),
  \end{equation}
  where $C$ depends on $\Omega$, $\Omega_*$, $a$, $\alpha$, $\left\Vert f_i\right\Vert _{\mathcal{C}^{0,\alpha}(\overline{\Omega})}$, and $\left\Vert g_i\right\Vert _{\mathcal{C}^{0,\alpha}(\overline{\Omega}_*)}$.
\end{theorem}

After completing this manuscript, we became aware of the recent preprint \cite{balakManole2025stabilityOT}, which establishes a stronger version of \cref{thm:d2_d2}. Their result requires a uniformly convex Kantorovich potential for only one of the two pairs of densities, whereas our proof assumes this condition on both. The argument in \cite{balakManole2025stabilityOT} exploits optimality directly, expressing the duality gap in terms of a fixed Kantorovich potential satisfying the uniform convexity condition. Our approach, by contrast, is based on PDE techniques and energy estimates.

Note that the setting of Theorem \ref{thm:d2_d2} rules out discrete measures. However, it is known that, if $f_1=f_0=\rho$ and $g_0$, $g_1$ are allowed to be discrete measures, then one \emph{cannot} obtain a bound of the form $\left\Vert \nabla\phi_{1}-\nabla\phi_{0}\right\Vert _{L^{2}(\Omega)}\leq C\mk_1(g_0,g_1)^{\alpha}$ for $\alpha>\frac{1}{2}$, which is currently conjectured as the optimal exponent (see \cite[Sec. 4]{Gigli2011HolderContCurve} and \cite[Lemma 5.2]{Delalande2022Thesis}). The bound on the optimal plans themselves \eqref{eq:boundplans} is also known to be false for non-smooth densities, see \cite{ford2025stabOT}. In this sense, the Lipschitz nature of \cref{eq:L2_d2_bound,eq:boundplans} seems to be a special feature of regular measures.

We also provide a stability estimate on the Brenier potentials when the continuity assumption is removed. However, we consider the stronger $L^2$ norm as a right hand side.
\begin{theorem}\label{thm:L2_L2_stab_potential}
    Let $\Omega$, $\Omega_*\subset \mathbb{R}^d$ be bounded, uniformly convex sets with boundaries of class $\mathcal C^{2,\alpha}$.
    Let $f_0$, $f_1\in L^{2}(\Omega)$, $g_0,g_1\in L^{2}(\Omega_{*})$ be probability densities with $f_i$, $g_i\geq a$, $f_i\leq A$ and let $\phi_{i}$ be Brenier potentials from $f_i$ to $g_i$, normalized so that
  \begin{equation*}
    \int_{\Omega}e^{-\phi_{i}(x)}\,dx=1.
  \end{equation*}
  Then we have 
  \begin{equation}\label{eq:L2_L2}
    \left\Vert \phi_{0}-\phi_{1}\right\Vert _{L^{2}(\Omega)}\leq C\left(\left\Vert f_0-f_1\right\Vert _{L^{2}(\Omega)}+\left\Vert g_0-g_1\right\Vert _{L^{2}(\Omega_{*})}\right)
  \end{equation}
  and
  \begin{align}\label{eqn: L2 grad bound}
      \lVert \nabla \phi_0-\nabla\phi_1\rVert_{L^2(\Omega)}
      \leq C\left(\left\lVert f_0-f_1\right\rVert^{\frac{1}{3}}_{L^{2}(\Omega)}+\left\lVert g_0-g_1\right\rVert^{\frac{1}{3}} _{L^{2}(\Omega_{*})}\right)
  \end{align}
  for some constant $C$ depending on $d$, $\textnormal{diam}(\Omega)$, $\textnormal{diam}(\Omega_*)$, $a$, and $A$.
\end{theorem}
The main feature that differentiates (\ref{eq:L2_L2}) from other estimates in the literature is again its Lipschitz nature. However, it is again limited to non-degenerate densities. We believe that it may be possible to improve the previous estimate to the bound
\begin{equation}\label{eq:dream_Linfty_bound}
\left\Vert \phi_{0}-\phi_{1}\right\Vert _{L^{\infty}(\Omega)}\leq C\left(\left\Vert f_0-f_1\right\Vert _{L^{2}(\Omega)}+\left\Vert g_0-g_1\right\Vert _{L^{2}(\Omega_{*})}\right)
\end{equation}
while retaining the same hypotheses. This would, in fact, hold if it were possible to prove a Sobolev inequality of the form
\begin{equation*}
  \left(\int_\Omega\lvert f\rvert^{2^{*}}\right)^{\frac{1}{2^{*}}}\leq C_{\textnormal{sob}}
  \left[
  \left(\int_\Omega\inner{(D^2\phi)^{-1}\nabla f}{\nabla f}\right)^\frac{1}{2} +
  \left(\int_\Omega |f|^{2}\right)^\frac{1}{2}\right],
  \qquad \forall f \in \mathcal C^1(\overline \Omega),
\end{equation*}
for a constant which depends only on $\Omega$, the upper and lower bounds of $\det D^2\phi$. Then (\ref{eq:dream_Linfty_bound}) would follow from \cref{lem:moser_estimate}, which would be proven for the coefficient matrix $(D^2\phi)^{-1}$ through a Moser iteration. Note that such a Sobolev inequality was proven in \cite{tian_wang2008MA_sob_ineq} under the boundary condition $ f\vert_{\partial\Omega}\equiv 0$. Finally, we observe that (\ref{eqn: L2 grad bound}) is a direct consequence of (\ref{eq:L2_L2}) and Proposition 4.1 of \cite{MerigotDela2023StabOt}.

Our third stability result is in a uniform metric.

\begin{theorem}\label{thm:unif_stab}
  Let $\Omega$, $\Omega_*\subset \mathbb{R}^d$ be bounded uniformly convex sets with boundaries of class $\mathcal C^{2,\beta}$. Also suppose $f_0$, $f_1\in\mathcal{C}^{0,\beta}(\overline{\Omega})$ and $g_0$, $g_1\in\mathcal{C}^{0,\beta}(\overline{\Omega}_*)$ are probability densities with $f_i$, $g_i\geq a>0$. Finally, let $\phi_i$ be Brenier potentials from $f_i$ to $g_i$ with
  \begin{align*}
      \int_{\Omega}\phi_i(x)\,dx=0.
  \end{align*}
  Then for any $\alpha<\beta$, we have
  \begin{equation*}
    \left\Vert \phi_{1}-\phi_{0}\right\Vert _{\mathcal C^{1,\alpha}(\overline \Omega)}
    \leq C\left(\left\Vert f_1-f_0\right\Vert _{L^{p}(\Omega)}+\left\Vert g_1-g_0\right\Vert _{L^{p}(\Omega_{*})}\right),
  \end{equation*}
  where $p = \frac{d}{1-\alpha}$ and $C$ depends on $\Omega$, $\Omega_*$, $a$, $\alpha$, $\beta$, $\left\Vert f_i\right\Vert _{\mathcal{C}^{0,\beta}(\overline{\Omega})}$, and $\left\Vert g_i\right\Vert _{\mathcal{C}^{0,\beta}(\overline{\Omega}_*)}$.
\end{theorem}

Quantitative stability results that uniformly estimate the map or the potential are very scarce. At the moment, we are only aware of two other results in this direction, \cite[Theorem 5.1]{KitaWarren2012OTSphere} and \cite[Proposition 5.1]{JeongKitagawa25}. Both results provide an interior $L^\infty$ estimate of the optimal map around the identity (i.e. setting $f_0=g_0$) in terms of the $2$-Monge--Kantorovich distance with an exponent that decays algebraically in the dimension, but for measures supported on a codimension one surface. Theorem \ref{thm:unif_stab} differs from the results of \cite{KitaWarren2012OTSphere, JeongKitagawa25} in that (i) we can estimate perturbations around maps other than the identity, (ii) the bounds do not have exponents that degenerate with the dimension, and (iii) the estimate is global. On the other hand, our results apply to absolutely continuous measures, rather than ones supported on a submanifold.

\begin{remark}\label{rmk:sharpness}
    If the assumptions from Theorem \ref{thm:unif_stab} hold for some $0<\alpha <1$, then we obtain the same estimate for any $0<\overline{\alpha} <\alpha$, showing that for any $p>d$ we have
    \begin{equation}\label{eq:Linfty_Lp_stab}
        \left\Vert \nabla\phi_{1}-\nabla\phi_{0}\right\Vert _{L^{\infty}(\Omega)}\leq C\left(\left\Vert f_0-f_1\right\Vert _{L^{p}(\Omega)}+\left\Vert g_0-g_1\right\Vert _{L^{p}(\Omega_{*})}\right),
    \end{equation}
    where $C$ depends on $\Omega$, $\Omega_*,a,p,\left\Vert f_i\right\Vert _{\mathcal{C}^{0,\alpha}(\overline{\Omega})}$, and $\left\Vert g_i\right\Vert _{\mathcal{C}^{0,\alpha}(\overline{\Omega}_*)}$. Of the hypotheses of Theorem \ref{thm:unif_stab}, the most fundamental one for (\ref{eq:Linfty_Lp_stab}) to hold seems to be the lower bound $g_i\geq a$. Indeed, if $g_i$ is allowed to vanish at even a single point, it can be shown there is no pair of constants $0<\eta<1$ and $C>0$ such that
    \begin{equation*}
        \left\Vert \nabla\phi_{1}-\nabla\phi_{0}\right\Vert _{L^{\infty}(\Omega)}\leq C\left(\left\Vert f_0-f_1\right\Vert _{L^{p}(\Omega)}^\eta+\left\Vert g_0-g_1\right\Vert _{L^{p}(\Omega_{*})}^\eta\right).
    \end{equation*}
    It is not clear to us which of the other assumptions from Theorem \ref{thm:unif_stab} may be relaxed. In fact, we note that in dimension $1$, given any densities $f_0$, $f_1\in\mathcal{P}(I)$, $g_0,g_1\in\mathcal{P}(J)$ on intervals $I,J\subset \mathbb R$ with $g_0,g_1\geq a$, we have
    \begin{equation}\label{eq:1D_Linfty_L1}
        \left\Vert \phi_{1}'-\phi_{0}'\right\Vert _{L^{\infty}(I)}\leq\frac{1}{a}\left(\left\Vert f_1-f_0\right\Vert _{L^{1}(I)}+\left\Vert g_1-g_0\right\Vert _{L^{1}(J)}\right).
    \end{equation}
    See the end of Section \ref{sec:uniform_stability_potentials} for a proof of these statements.
\end{remark}

The main technical ingredient in the proofs of \cref{thm:d2_d2,thm:L2_L2_stab_potential,thm:unif_stab} is to show the regularity of the Brenier potential along smooth perturbations of the source and target measure, see \cref{thm:implicit_fun} for a precise statement. As a further application of this technical tool, we prove an explicit formula for the second variation of the $2$-Monge--Kantorovich distance. This formula is of independent interest, as it quantifies explicitly the linear convexity of the $2$-Monge--Kantorovich distance.
\begin{theorem}\label{thm:2var_d2}
  Let $\Omega$, $\Omega_*\subset \mathbb{R}^d$ be bounded uniformly convex open sets with $\mathcal C^{3,\alpha}$ boundary. Consider $f$, $h\in \mathcal C^{0,\alpha}(\overline \Omega)$, $g$, $k\in \mathcal C^{0,\alpha}(\overline \Omega_*)$ where $f$, $g\geq a>0$ and
  \begin{equation*}
    \int_{\Omega} h(x)f(x)\,dx = \int_{\Omega_*} k(y)g(y)\,dy =0,
  \end{equation*}
  so that $f_t := f(1+th)$, $g_t :=g(1+tk)$ are probability densities for $t$ small enough. If $\phi_t$ is the Brenier potential from $f_t$ to $g_t$ with integral zero, then
  \begin{equation}\label{eq:second_var_thm}
    \frac{1}{2}\left.\frac{d^{2}}{dt^{2}}\right|_{t=0}\left[\mk_2(f_{t},g_{t})^2\right]
    =\int_{\Omega}\inner{\left(D^{2}\phi_0(x)\right)^{-1}\nabla\xi(x)}{\nabla\xi(x)} f(x)\,dx,
  \end{equation}
  where $\xi=\left.\partial_{t}\right|_{t=0}\phi_t$ is weak solution of
  \begin{equation}\label{eq:lin_MA_2var_d2}
    \begin{cases}
      -\div\left[f(D^{2}\phi_0)^{-1}\nabla\xi\right]=\left(k\circ \nabla\phi_0-h\right)f, & \textnormal{in }\Omega,\\
       \inner{(D^{2}\phi_0)^{-1}\nabla\xi}{\normal} =0, & \textnormal{on }\partial\Omega,
    \end{cases}
  \end{equation}
  and $\normal$ denotes the unit outer normal of $\Omega$.
\end{theorem}
To our knowledge, the only instance in the literature of the second variation of the Monge--Kantorovich distance is found in \cite{otto_vill2000generalization}, where some formal calculations are made by perturbing around the identity mapping. In this work, we extend and justify these calculations for general optimal transport maps between regular densities.

We conclude this section by highlighting \cref{eq:lin_MA_2var_d2}, which characterizes the variation $\xi=\left.\partial_t\right|_{t=0}\phi_t$. It is a form of the linearized Monge--Ampère equation, written in a divergence structure that differs from the one found in the literature. To the best of our knowledge, this is the first time this equation is cast in such a form in the optimal transport setting. This reformulation, which we state in more generality in \cref{thm:implicit_fun,cor:implicit_fun}, is a key ingredient in the proofs of our main results.

\subsection{Main idea}\label{sec:intuition}
Given two pairs of source and target probability densities $(f_0,g_0)$ and $(f_1,g_1)$, we fix a suitable interpolation curve $\{(f_t,g_t)\}_{0\leq t\leq 1}$ and consider an associated curve $\{\phi_t\}_{0\leq t\leq 1}$ consisting of Brenier potentials from $f_t$ to $g_t$,
which uniquely determines $\phi_t$ up to constants in our setting. To control the difference between the two potentials, we can express it as an integral of its derivative along the interpolation path, i.e. 
$$\phi_1-\phi_0=\int_0^1 \partial_t \phi_t \, dt.$$ 
To obtain a useful estimate we characterize the derivative $\partial_t \phi_t$ via linearization of the Monge--Amp{\`e}re equation, which is the content of the following theorem. This idea of controlling the potentials using the linearized equation is already present in the earlier work of Loeper \cite{loeper2005regularity}.
\begin{theorem}\label{thm:implicit_fun}
  Let $\Omega$, $\Omega_*\subset \mathbb R^d$ be uniformly convex with boundaries of class $\mathcal C^{1,1}$ and $\mathcal C^{m+1,\alpha}$ respectively for some $m\geq 1$. Also suppose $t\in[0,1]\mapsto f_t\in \mathcal C^{1,\alpha}(\overline \Omega)$, $t\in [0,1]\mapsto g_t\in \mathcal C^{m,\alpha}(\overline \Omega_*)$ are $\mathcal C^m$ as maps between Banach spaces, where $f_t$ and $g_t$ are probability densities and there exists $a>0$ such that $f_t$, $g_t\geq a$ on $\Omega$ and $\Omega_*$ respectively. Then if $\phi_t$ is the Brenier potential from $f_t$ to $g_t$ for each $t$ satisfying $\int_\Omega \phi_t= 0$, for each $0<\beta <\alpha$, we have that the mapping
  \begin{equation*}
    t\mapsto \phi_t\in \mathcal C^{2,\beta}(\overline \Omega)
  \end{equation*}
  is $\mathcal C^m$ as a curve in $\mathcal C^{2,\beta}(\overline \Omega)$. Moreover, for each $t$, $\xi_t = \partial_t \phi_t$ is the unique solution to
  \begin{equation}\label{eq:boundary_val_prob_implicit}
    \begin{cases}
      \displaystyle{-\div\left[f_t(D^{2}\phi_t)^{-1}\nabla\xi_t\right]=\left(\frac{\partial_{t}g_t\left(\nabla\phi_t\right)}{g_t\left(\nabla\phi_t\right)}-\frac{\partial_{t}f_t}{f_t}\right)f_t,} &
      \textnormal{in }\Omega \\ 
      {\displaystyle \inner{(D^{2}\phi_t)^{-1}\nabla\xi_t}{\normal} =0,} & \textnormal{on }\partial\Omega\\
      \int_{\Omega}\xi_t=0.
    \end{cases}
  \end{equation}
\end{theorem}
\Cref{thm:implicit_fun} is proved using an appropriate version of the inverse function theorem. The use of the implicit function theorem in conjunction with the Monge--Amp\`ere equation is not new; for instance, Urbas employed this approach in \cite{Urbas1995,Urbas1998,Urbas1986} to construct solutions to related equations via the continuity method. Our use is different in nature: rather than establishing existence, our objective is to characterize and estimate the variation $\partial_t\phi_t$, which forms the basis for the stability estimates developed in this paper. More recently, a similar result was established in \cite{gonzalez2024_linearization_MA} for the case $m=1$, whose argument serves as the main inspiration for the proof presented in Section~\ref{sec:time_reg_ifc}. 

A key novelty is the divergence-form formulation \cref{eq:boundary_val_prob_implicit}. This structure substantially simplifies the analysis. Moreover, using elliptic regularity and an approximation argument, it allows us to establish the differentiability of $\phi_t$ in the parameter $t$ without the assumption $g\in\mathcal C^{1,\alpha}$ required in \cite{gonzalez2024_linearization_MA}.
\begin{corollary}\label{cor:implicit_fun}
    Consider $\Omega$, $\Omega_*\subset \mathbb R^d$ uniformly convex with boundaries of class $\mathcal C^{1,1}$, and $\mathcal C^{2,\alpha}$ respectively. Assume $f_t$, $g_t$ satisfy:
    \begin{enumerate}
        \item There is $a>0$ such that $f_t$, $g_t\geq a$ on $\Omega$ and $\Omega_*$ respectively.
        \item $f_t\in \mathcal C^{0,\alpha}(\overline \Omega)$, $g_t\in \mathcal C^{0,\alpha}(\overline \Omega_*)$ with $\Vert f_t\Vert_{\mathcal C^{0,\alpha}(\overline \Omega)}$, $\Vert g_t\Vert_{\mathcal C^{0,\alpha}(\overline \Omega_*)}$ uniformly bounded in $t$.
        \item\label{cond: unif cont partial_t} Both $f_t$ and $g_t$ are differentiable in $t$ for $0<t<1$, and $\partial_tf_t$, $\partial_tg_t$ are continuous with respect to $(t,x)\in [0,1]\times \overline \Omega$, $(t,y)\in [0,1]\times \overline \Omega_*$.
    \end{enumerate}
    If we let $\phi_t$ be the Brenier potential from $f_t$ to $g_t$ with $\int_\Omega\phi_t\,dx=0$, then $\phi_t$ is differentiable in $t$ for $t\in(0,1)$, and $\xi_t :=\partial_t\phi_t$ is a weak solution of (\ref{eq:boundary_val_prob_implicit}) for each $t\in (0, 1)$. That is, for every $\theta\in H^1(\Omega)$,
    \begin{align}\label{eqn: weak PDE nonsmooth}
        \int_\Omega \inner{(D^{2}\phi_t)^{-1}\nabla\xi_t}{\nabla \theta}f_tdx=
        \int_\Omega\left(\frac{\partial_{t}g_t\left(\nabla\phi_t\right)}{g_t\left(\nabla\phi_t\right)}-\frac{\partial_{t}f_t}{f_t}\right)\theta f_tdx.
    \end{align}
\end{corollary}

Our results from the previous subsection follow by estimating $\partial_t \phi_t$ using Theorem~\ref{thm:implicit_fun} or Corollary~\ref{cor:implicit_fun} and integrating in the $t$ variable; in all of  \cref{thm:d2_d2,thm:unif_stab,thm:L2_L2_stab_potential,thm:2var_d2}, we make an interpolation or perturbation which is linear in the parameter $t$. For Theorems~\ref{thm:d2_d2} and~\ref{thm:unif_stab} our densities are assumed to be H{\"o}lder, hence we can directly invoke Corollary~\ref{cor:implicit_fun}. In Theorem~\ref{thm:L2_L2_stab_potential} our densities are $L^2$ regular, hence we apply Theorem~\ref{thm:implicit_fun} to smooth densities and show the estimates are independent of smoothness. For Theorem~\ref{thm:2var_d2}, we again start with more regular densities, but must undertake a more careful approximation argument to obtain the second variation formula for H\"older densities. We note that we do not invoke the full strength of Corollary~\ref{cor:implicit_fun}: we only use it for the reduction in required regularity for $f_i$ and $g_i$. Moreover, it shows that the result also applies when the densities $f_t$, $g_t$ only have pointwise differentiability as functions in $(t, x)$  (i.e., even if they are not known to be Fr\'echet differentiable as curves into Banach spaces).

The proofs of \cref{thm:d2_d2,thm:unif_stab} rely on the following estimate.
\begin{lemma}
\label{lem: global C2 estimate}
Let $\Omega$, $\Omega_{*}\subset\mathbb{R}^{d}$ with boundaries of class $\mathcal{C}^{2,\beta}$ for some $\beta\in (0, 1]$. Consider densities $f\in \mathcal C^{0,\beta}(\overline\Omega)$, $g\in \mathcal C^{0,\beta}(\overline \Omega_*)$ with $f$, $g\geq a$ for some constant $a>0$ on $\Omega$ and $\Omega_{*}$ respectively, and let $\phi$ be the convex solution of
  \begin{equation}\label{eq:MA_appendix}
    \det D^{2}\phi=\frac{f(x)}{g\big(\nabla\phi(x)\big)}
    ,\qquad\nabla\phi(\Omega)=\Omega_{*}
    ,\qquad \int_\Omega \phi = 0.
  \end{equation}
  Then for any $\alpha<\beta$, it holds $\Vert \phi\Vert _{\mathcal{C}^{2,\alpha}(\overline{\Omega})}\leq C$, where $C$ depends on $\Omega$, $\Omega_{*}$, $\alpha$, $\beta$, $a$, $\left\Vert f\right\Vert _{\mathcal{C}^{0,\beta}(\overline{\Omega})}$ and $\left\Vert g\right\Vert _{\mathcal{C}^{0,\beta}(\overline{\Omega}_{*})}$.
\end{lemma}
This allows us to quantify the ellipticity of (\ref{eq:boundary_val_prob_implicit}); we provide the proof of this estimate in Appendix \ref{sec:compactness_ellipticity}. The authors have recently learned of \cite{collins2025boundaryRegOT} which can be used to replace Lemma~\ref{lem: global C2 estimate} under weaker conditions on the densities and domains, however the presentation here is kept this way for the purpose of self-containedness.

\begin{remark} \label{rmk:ellipticity_D2phi}
Recall it is standard that if a solution $\phi$ of~\eqref{eq:MA_appendix} has a bound on $\lVert \phi\rVert_{\mathcal{C}^2(\overline\Omega)}$, then it is actually uniformly convex. More specifically, first if $\Lambda:=\lVert D^{2}\phi\rVert _{L^{\infty}(\Omega)}$ then clearly $D^2\phi \leq \Lambda \Id$. On the other hand, if $0\leq\lambda_{1}(x),\ldots,\lambda_{d}(x)\leq\Lambda$ are the eigenvalues of $D^{2}\phi$ then
  \begin{align*}
    \lambda_{1}(x)\cdots\lambda_{d}(x)&=\det D^{2}\phi(x)=\frac{f(x)}{g\left(\nabla\phi(x)\right)}\\
    \implies&\lambda_{i}(x)=\frac{f(x)}{g\left(\nabla\phi(x)\right)}\prod_{j\neq i}\lambda_{j}(x)^{-1}\geq\frac{\inf f}{\sup g}\Lambda^{-d+1}.
  \end{align*}
We shall use this fact freely in the sequel.
\end{remark}

Regarding Theorem \ref{thm:2var_d2}, once $\phi_t$ is known to be differentiable in $t$, the expression (\ref{eq:second_var_thm}) can be obtained by the differentiation
$$
\left.\frac{d^2}{dt^2}\right|_{t=0} \int_\Omega|\nabla\phi_t(x)-x|^2f_t(x)\,dx.
$$

\textbf{Organization of the paper.} The proofs of the theorems have been divided into sections. Theorem \ref{thm:d2_d2} is proven in Section \ref{sec:stab_w2_metrics}, Theorem \ref{thm:L2_L2_stab_potential} is proven in Section \ref{sec:stab_L2}, Theorem \ref{thm:unif_stab} is proven in Section \ref{sec:uniform_stability_potentials}, Theorem \ref{thm:2var_d2} is proven in Section \ref{sec:second_variation}, and Theorem \ref{thm:implicit_fun} together with Corollary \ref{cor:implicit_fun} are proven in Section \ref{sec:time_reg_ifc}. Appendix \ref{sec:reg_comp} contains a simplified summary of the main results of \cite{llave1999_reg_comp}, which has been included for ease of reference and for the reader's convenience. Appendix \ref{sec:compactness_ellipticity} contains a compactness argument used to prove Lemma \ref{lem: global C2 estimate}, which is vital for the proof of \cref{thm:d2_d2,thm:unif_stab}.

\section{Proof of \cref{thm:d2_d2}}\label{sec:stab_w2_metrics}
For the remainder of the paper we fix domains $\Omega$ and $\Omega_*\subset \R^d$ and probability densities $f_0$, $f_1$ supported on $\overline{\Omega}$ and $g_0$, $g_1$ supported on $\overline{\Omega}_*$. In this section, we assume $\Omega$ and $\Omega_*$ are uniformly convex. For $\Omega_*$, we fix $\omega_*: \R^d\to \R$ to be a convex \emph{defining function} for $\Omega_*$, that is, a convex function such that
\begin{equation*}
  \Omega_* = \{y\in \mathbb R^d\mid\omega_*(y)<0\},\qquad |\nabla \omega_*(y)|\equiv 1\,\textnormal{ on }\partial \Omega_*;
\end{equation*}
in particular $\nabla \omega_*(y)=\normal_*(y)$ is the outer unit normal. Note if $\phi_i$ is a Brenier potential from $f_i$ to $g_i$ and  $\nabla \phi_i$ is continuous on $\partial \Omega$, the boundary condition
\begin{align}\label{eqn: MA bdry cond}
    \omega_*(\nabla \phi_i(x))=0
\end{align}
holds for all $x\in \partial \Omega$. Additionally, under the hypotheses of Theorems~\ref{thm:d2_d2}, \ref{thm:unif_stab}, \ref{thm:2var_d2}, and \ref{thm:implicit_fun}, and Corollary~\ref{cor:implicit_fun}, by  \cite{caffarelli1992reg_convex_potential} we see that $\phi_i$ is strictly convex, meaning that $\nabla \phi_i$ is invertible on $\Omega$. We will freely use this fact.

\begin{remark}
  A natural approach to prove \cref{thm:d2_d2} would be to interpolate along $2$-Monge--Kantorovich geodesics between $(f_0,g_0)$ and $(f_1,g_1)$. Such geodesics satisfy continuity equations
  $$
  \partial_tf_t + \textnormal{div}(f_t\nabla u_t)=0,\qquad \partial_tg_t + \textnormal{div}(g_t\nabla v_t)=0,
  $$
  where the vector field potentials $u_t,v_t$ verify
  $$
  \int_0^1\int_\Omega|\nabla u_t|^2f_t\,dx\,dt=\mk_2(f_0,f_1)^2,\qquad
  \int_0^1\int_{\Omega_*}|\nabla v_t|^2g_t\,dx\,dt=\mk_2(g_0,g_1)^2,
  $$
  yielding the right-hand side of \cref{thm:d2_d2}. However, even when $f_0$, $f_1$ are smooth and supported on $\overline\Omega$, the intermediate densities $f_t$ along a $d_2$ geodesic may not be supported on the whole $\overline\Omega$. Most importantly, the support of $f_t$ may fail to be convex as shown in \cite{SantambrogioWang2016Counterexamples}, which means that  \cref{cor:implicit_fun} cannot be applied. Instead, we use linear interpolations.
\end{remark}

To prove \cref{thm:d2_d2} we use the following estimate relating an $H^{-1}$ norm and the 2-Monge-Kantorovich distance. These estimates are originally due to \cite{Peyre2018Hminus1W2}, although we prove the following particular case for the sake of completeness:

\begin{lemma}\label{lem:Hminus1}
  Consider densities $f_0,f_1$ supported in $\Omega$ with $f_0,f_1 \leq A$. Then we have
  \begin{equation*}
    \Vert f_1-f_0\Vert_{H^{-1}_0(\Omega)} \leq \sqrt A \mk_2(f_0,f_1),
  \end{equation*}
  where we define
  \begin{equation*}
    \Vert h\Vert_{H^{-1}_0(\Omega)} = \sup\left\{
      \int_\Omega h(x)\theta(x)\,dx \mid \theta \in H^1(\Omega), \int_\Omega |\nabla\theta(x)|^2\,dx \leq 1
    \right\},
  \end{equation*}
  for any function $h$ with average $0$.
\end{lemma}
\begin{proof}
  Let $\zeta_1(x)$ be a Brenier potential between $f_0$ and $f_1$. Then the $\mk_2$ geodesic from $f_0$ to $f_1$ can be written as
  $$
  f_t = (\nabla \zeta_t)_\# f_0, \qquad \zeta_t(x) := (1-t)\frac{1}{2}|x|^2 + t\zeta_1(x).
  $$
  It is well known that for any $1\leq p<\infty$ the functional $t\mapsto \int_\Omega f_t^p\,dx$ is convex in $t$. This immediately shows that
  $$
  \Vert f_t\Vert_{L^p(\Omega)} \leq \max\{\Vert f_0\Vert_{L^p(\Omega)},\Vert f_1\Vert_{L^p(\Omega)}\} \leq A,
  $$
  and letting $p\to\infty$ yields $ f_t\leq A$. Then, given $\theta\in \mathcal C^1(\overline \Omega)$ we have
  \begin{align*}
    \int_\Omega (f_1(x)-f_0(x))\theta(x)\,dx & = \int_\Omega \left(\theta(\nabla\zeta_1(x))-\theta(x)\right)f_0(x)\,dx\\
    & = \int_0^1\frac{d}{dt}\int_\Omega \theta(\nabla\zeta_t(x))f_0(x)\,dx\,dt\\
    & = \int_0^1\int_\Omega \inner{\nabla\theta(\nabla\zeta_t(x))}{\nabla\zeta_1(x)-x}f_0(x)\,dx\,dt\\
    & \leq \mk_2(f_0,f_1)\int_0^1\left(\int_\Omega |\nabla\theta(\nabla\zeta_t(x))|^2f_0(x)\,dx\right)^{\frac{1}{2}}\,dt\\
    & =\mk_2(f_0,f_1)\int_0^1\left(\int_\Omega |\nabla\theta(x)|^2f_t(x)\,dx\right)^{\frac{1}{2}}\,dt \\
    & \leq \sqrt{A}\mk_2(f_0,f_1)\Vert\nabla\theta\Vert_{L^2(\Omega)},
  \end{align*}
  which finishes the proof.
\end{proof}

Now we are ready to prove our first stability result.

\begin{proof}[Proof of \cref{thm:d2_d2}]
  Consider linear interpolations
  \begin{align}\label{eq:linear_interp_d2}
    f_t=(1-t)f_0+tf_1,\qquad g_t=(1-t)g_0+tg_1,\qquad t\in[0,1].
  \end{align}
  If we let $\phi_t$ be the Brenier potential between $f_t$, $g_t$, then Corollary~\ref{cor:implicit_fun} implies that $\phi_t$ is differentiable with respect to $t$ with $\xi_t=\partial_t\phi_t$ weakly satisfying
   \begin{equation}\label{eq:PDE_d2_d2}
    \begin{cases}
    {\displaystyle-\div\left[f_t\big(D^{2}\phi_t\big)^{-1}\nabla \xi_t\right]=\frac{g_1(\nabla\phi_t)-g_0(\nabla\phi_t)}{g_{t}(\nabla\phi_t)}f_{t}-(f_1-f_0),} & \textnormal{in }\Omega,\\
    \inner{f_t\big(D^{2}\phi_t\big)^{-1}\nabla\xi_t}{\normal} =0, & \textnormal{on }\partial\Omega.
    \end{cases}
  \end{equation}
  Since $\Omega$ and $\Omega_*$ are uniformly convex and $f_t$, $g_t$ satisfy uniform H\"older and lower bounds in $t$, \cref{lem: global C2 estimate} yields a uniform ellipticity bound
  \begin{equation*}
    \Lambda^{-1}\Id\leq(D^2\phi_t)^{-1}\leq\Lambda\Id,
  \end{equation*}
  where $\Lambda$ depends only on $\Vert f_i\Vert_{\mathcal C^{0,\alpha}(\overline \Omega)}$, $\Vert g_i\Vert_{\mathcal C^{0,\alpha}(\overline \Omega_*)}$, $\Omega$, $\Omega_*$, and $a$. Now we take $\xi_t$ as a test function for \cref{eq:PDE_d2_d2}, and use Lemma~\ref{lem:Hminus1} to obtain for $A:=\max_{i=0,1}(\Vert f_i\Vert_{\mathcal C^{0,\alpha}(\overline \Omega)}, \Vert g_i\Vert_{\mathcal C^{0,\alpha}(\overline \Omega_*)})$,
  \begin{align}
    &\int_\Omega \left\langle(D^2\phi_t)^{-1}\nabla\xi_t,\nabla\xi_t\right\rangle f_t\,dx  \nonumber
    = \int_{\Omega_*} (g_1-g_0) \xi_t(\nabla\psi_t) \,dy - \int_{\Omega}(f_1-f_0)\xi_t\,dx\\
    &\leq \lVert g_0-g_1\rVert_{H^{-1}_0(\Omega_*)}\lVert \nabla [\xi_t\circ \nabla \psi_t]\rVert_{L^2(\Omega_*)}+
    \lVert f_0-f_1\rVert_{H^{-1}_0(\Omega)}\lVert \nabla \xi_t \rVert_{L^2(\Omega)}\notag\\
    & \leq \sqrt{A}\mk_2(g_0,g_1) \Vert (D^2\psi_t)\nabla\xi_t(\nabla\psi_t)\Vert_{L^2(\Omega_*)}
     + \sqrt{A}\mk_2(f_0,f_1) \Vert \nabla\xi_t\Vert_{L^2(\Omega)},\label{eq:d2_d2_estimate}
  \end{align}
  where $\psi_t$ is a Brenier potential from $g_t$ to $f_t$, satisfying $\nabla \psi_t=(\nabla \phi_t)^{-1}$ a.e. on $\Omega$.
  Now we note that
  \begin{align*}
    \Vert (D^2\psi_t)\nabla\xi_t(\nabla\psi_t)\Vert_{L^2(\Omega_*)}^2 & 
    \leq a^{-1}\int_{\Omega_*}|(D^2\psi_t)\nabla\xi_t(\nabla\psi_t)|^2g_t\,dy\\
    & = a^{-1}\int_{\Omega}|(D^2\phi_t)^{-1}\nabla\xi_t|^2f_t\,dx \leq \frac{A\Lambda^2}{a}\Vert \nabla\xi_t\Vert_{L^2(\Omega)}^2.
  \end{align*}
  Plugging this into \cref{eq:d2_d2_estimate}, using the bounds on $(D^2\phi_t)^{-1}$ and a weighted Young's inequality,
  \begin{align*}
    a\Lambda^{-1} \Vert \nabla\xi_t\Vert_{L^2(\Omega)}^2  & \leq 
    \int_{\Omega_*} (g_1-g_0) \xi_t(\nabla\psi_t) \,dy - \int_{\Omega}(f_1-f_0)\xi_t\,dx\\
    & \leq \frac{A\Lambda}{\sqrt{a}}\mk_2(g_0,g_1)\Vert \nabla\xi_t\Vert_{L^2(\Omega)}
     + \sqrt{A}\mk_2(f_0,f_1) \Vert \nabla\xi_t\Vert_{L^2(\Omega)}\\
    & \leq \frac{a\Lambda^{-1}}{2}\Vert \nabla\xi_t\Vert_{L^2(\Omega)}^2 + 
    C\left(\mk_2(f_0,f_1)^2+\mk_2(g_0,g_1)^2\right),
  \end{align*}
  In particular, for some $C>0$ depending only on $\Vert f_i\Vert_{\mathcal C^{0,\alpha}(\overline \Omega)}$, $\Vert g_i\Vert_{\mathcal C^{0,\alpha}(\overline \Omega_*)}$, $\Omega$, $\Omega_*$, and $a$,
  $$
  \Vert \nabla\xi_t\Vert_{L^2(\Omega)}^2 
  \leq 
   C
   \left(\mk_2(f_0,f_1)^2+\mk_2(g_0,g_1)^2\right).
  $$
  Then the stability estimate follows from
  \begin{align*}
      \left\Vert \nabla\phi_{1}-\nabla\phi_{0}\right\Vert _{L^{2}(\Omega)} 
      & =\left\Vert \int_{0}^{1}\nabla\xi_t\,dt\right\Vert _{L^{2}(\Omega)}\\
      & \leq\int_{0}^{1}\left\Vert \nabla\xi_t\right\Vert _{L^{2}(\Omega)}dt
      \leq 
      C\left(\mk_2(f_0,f_1)+\mk_2(g_0,g_1)\right),
  \end{align*}
  finishing the proof of the estimate \cref{eq:L2_d2_bound}.
  
  Next, we show that \cref{eq:boundplans} follows directly from \cref{eq:L2_d2_bound}; recall $\pi_i$ is an optimal plan between $f_i$ and $g_i$ for $i=0, 1$. Let $\gamma \in \mathcal{P}(\Omega^2)$ be an optimal plan from $f_0$ to $f_1$ and define $\Pi\in \mathcal{P}\left((\Omega\times\Omega_*)^2\right)$ by
  $$
  \Pi := \Theta_\#\gamma,\qquad \Theta:\Omega^2\longrightarrow (\Omega\times\Omega_*)^2,\qquad
  \Theta(x_0,x_1) := (x_0,\nabla\phi_0(x_0),x_1,\nabla\phi_1(x_1)).
  $$
  Since $\pi_i=(\Id\times \nabla \phi_i)_\sharp f_i$ for $i=0$, $1$, the left and right marginals of $\Pi$ are $\pi_0$ and $\pi_1$ respectively, thus
  \begin{align*}
    \mk_2(\pi_0,\pi_1)^2 & \leq \int_{(\Omega\times\Omega_*)^2}\left(|x_0-x_1|^2+|y_0-y_1|^2\right)\,d\Pi(x_0,y_0,x_1,y_1)\\
    & = \int_{\Omega\times \Omega}\left(|x_0-x_1|^2+|\nabla\phi_0(x_0)-\nabla\phi_1(x_1)|^2\right)\,d\gamma(x_0,x_1)\\
    &=\mk_2(f_0,f_1)^2+\int_{\Omega\times \Omega}\lvert\nabla\phi_0(x_0)-\nabla\phi_1(x_1)\rvert^2\,d\gamma(x_0,x_1).
  \end{align*}
  Using that $\nabla\phi_i$ are $\Lambda$-Lipschitz,
  \begin{align*}
    \left(\int_{\Omega\times \Omega}|\nabla\phi_0(x_0)-\nabla\phi_1(x_1)|^2\,d\gamma(x_0,x_1)\right)^{\frac{1}{2}} & 
    \leq \left(\int_{\Omega\times \Omega}|\nabla\phi_0(x_0)-\nabla\phi_0(x_1)|^2\,d\gamma(x_0,x_1)\right)^{\frac{1}{2}}\\
    & + \left(\int_{\Omega\times \Omega}|\nabla\phi_0(x_1)-\nabla\phi_1(x_1)|^2\,d\gamma(x_0,x_1)\right)^{\frac{1}{2}}\\
    & \leq \Lambda \mk_2(f_0,f_1) + \sqrt A\Vert\nabla \phi_1-\nabla\phi_0\Vert_{L^2(\Omega)}.
  \end{align*}
  Now the stability estimate \cref{eq:L2_d2_bound} on $\left\Vert \nabla\phi_{1}-\nabla\phi_{0}\right\Vert_{L^{2}(\Omega)}$ finishes the proof.
\end{proof}

\section{Proof of \cref{thm:L2_L2_stab_potential}}\label{sec:stab_L2}

The proof is based on a version of the Brascamp--Lieb inequality \cite{Brascamp_Lieb1976}, which can be seen to hold for log-concave measures supported on a bounded, convex domain. Specifically, if $\mu_F := e^{-F(x)}\mathds{1}_\Omega dx$ where $\Omega\subset \R^d$ is bounded and convex, $F\in \mathcal{C}^2(\overline\Omega)$, and $D^2F(x)>0$ for all $x\in \Omega$, then for any $u\in \mathcal{C}^1(\Omega)$,
\begin{align}\label{eqn: brascamp-lieb}
    \operatorname{Var}_{\mu_F}(u):=\int_\Omega \left(u-\int u\;d\mu_F\right)^2\,d\mu_F
\leq
\int_\Omega \inner{(D^2F)^{-1}\nabla u}{\nabla u}\,d\mu_F.
\end{align}
This form of inequality has previously been used to study stability of optimal transport in \cite[Theorem 2.3]{MerigotDela2023StabOt}, where it is attributed to \cite[Corollary 1.3]{LePeutrec17} and \cite[Section 3.1.1]{KolesnikovMilman17}, which requires at least $\mathcal{C}^2$ regularity of $\partial\Omega$. However, it is possible to obtain this inequality from only convexity of $\Omega$ with no additional assumption of regularity by following the argument on \cite[p. 161]{CorderoErausquinKlartag12}. Indeed, in the notation there, both \cite[(20)]{CorderoErausquinKlartag12} and the integration by parts formula before it hold for $\varphi\in \mathcal{C}^\infty_c(\Omega)$ (the former by integrating the standard Bochner's identity for the operator $L$). Next by \cite[Theorem 8.3]{gilbarg_trudinger1977elliptic}, there exists a weak solution $w\in W^{1, 2}_0(\Omega)$ to the equation $\Delta w-\inner{\nabla F}{\nabla w}=u-\int_\Omega ud\mu_F$. Since $\inner{\nabla F}{\nabla w}\in L^2(\Omega; dx)$ and $\Omega$ is convex, by \cite[Theorem 3.2.1.2]{Grisvard11} there exists a strong solution $v\in W^{2, 2}_0(\Omega)$ of $\Delta v=\inner{\nabla F}{\nabla w}+u-\int_\Omega ud\mu_F$. By the comparison principle \cite[Theorem 8.2]{gilbarg_trudinger1977elliptic}, we see $v=w$ a.e. on $\Omega$. Thus there exists $\varphi\in \mathcal{C}^\infty_c(\Omega)$ such that $\Delta \varphi-\inner{\nabla F}{\nabla \varphi}-(u-\int_\Omega ud\mu_F)$ has arbitrarily small $L^2(\Omega; dx)$ norm, hence small $L^2(\Omega; \mu_F)$ norm, from the boundedness of $F$. Thus the argument can be concluded exactly as on \cite[p. 161]{CorderoErausquinKlartag12} to obtain \eqref{eqn: brascamp-lieb}.

\begin{proof}[Proof of \cref{thm:L2_L2_stab_potential}]
  Again the proof will follow from the intuition outlined in \cref{sec:intuition}. In this case, we consider the linear interpolation of the densities
  $$
  f_t=(1-t)f_0+tf_1,\qquad g_t=(1-t)g_0+tg_1.
  $$
  First assume the densities are $\mathcal{C}^{1, \alpha}$ regular, then $t\mapsto f_t$ and $t\mapsto g_t$ are clearly smooth as curves valued in $\mathcal{C}^{1, \alpha}(\overline{\Omega})$ and $\mathcal{C}^{1, \alpha}(\overline{\Omega}_*)$ respectively, so that we can apply Theorem \ref{thm:implicit_fun}. In particular, if $\tilde\phi_t$ is the Brenier potential from $f_{t}$ to $g_{t}$ with $\int_\Omega \tilde\phi_t=0$ for all $t\in [0, 1]$ we have that $\tilde\phi_t\in C^{2, \alpha}(\overline{\Omega})$, is differentiable with respect to $t$, and $\tilde \xi_t:=\partial_t\tilde\phi_t$ solves, in the weak sense (using that $f_t$ and $g_t$ are linear interpolations),
  \begin{equation}\label{eq:PDE_xi_weaker_metrics}
    \begin{cases}
    {\displaystyle-\div\left[f_t\big(D^{2}\tilde\phi_t\big)^{-1}\nabla \tilde\xi_t\right]=\frac{g_1(\nabla\tilde\phi_t)-g_0(\nabla\tilde\phi_t)}{g_{t}(\nabla\tilde\phi_t)}f_{t}-(f_1-f_0),} & \textnormal{in }\Omega,\\
    \inner{f_t\big(D^{2}\tilde\phi_t\big)^{-1}\nabla\tilde\xi_t}{\normal} =0, & \textnormal{on }\partial\Omega.
    \end{cases}
  \end{equation}
  Now define $\phi_t:=\tilde\phi_t+\log\left(\int_\Omega e^{-\tilde\phi_t}\right)$, then we see $\int_\Omega e^{-\phi_t(x)}\,dx=1$ and $\nabla\phi_t=\nabla\tilde\phi_t$. Since $\partial_t\tilde{\phi}_t$ is bounded on $\overline{\Omega}$ and $\int_\Omega e^{-\tilde\phi_t}\neq 0$, we can justify differentiating under the integral to see that $\phi_t$ is also differentiable in $t$, and if $\xi_t:=\partial_t\phi_t$ we have $\nabla \xi_t=\nabla \tilde\xi_t$. Thus~\eqref{eq:PDE_xi_weaker_metrics} holds with $\xi_t$ and $\phi_t$ replacing $\tilde\xi_t$ and $\tilde\phi_t$.

  Next, since $\nabla\phi_t(\Omega)\subset\Omega_{*}$, we see $\phi_t$ is uniformly Lipschitz independent of $t$, thus $\left\Vert \phi_t\right\Vert _{L^{\infty}(\Omega)}$ is bounded in terms of $\Omega$ and $\Omega_{*}$; in particular there are $c_{1}$ and $c_{2}$ depending on $\Omega$ and $\Omega_{*}$ such that $c_1\leq \rho_t(x):=e^{-\phi_t(x)}\leq c_2$. For the rest of the computation $C>0$ will denote a constant depending on $d$, $\textnormal{diam}(\Omega)$, $\textnormal{diam}(\Omega_*)$, $a$, and $A$ which may change from line to line.  For each $t\in [0, 1]$ we can apply~\eqref{eqn: brascamp-lieb} with $F=\phi_t$ and $u=\xi_t$ to obtain
  \begin{align*}
    \int_{\Omega}\inner{\left(D^{2}\phi_t\right)^{-1}\nabla\xi_t}{\nabla\xi_t} f_t 
    & \geq C^{-1}\int_{\Omega}\inner{\left(D^{2}\phi_t\right)^{-1}\nabla\xi_t}{\nabla\xi_t} \,d\rho_t\\
    & \geq C^{-1}\int_{\Omega}\left(\xi_t-\int_\Omega \xi_t d\rho_t\right)^{2}\,d\rho_t\geq C^{-1}\int_{\Omega}\left(\xi_t-\int_\Omega \xi_t d\rho_t\right)^{2}\,dx.
  \end{align*}
  Taking $\xi_t-\int_\Omega \xi_t d\rho_t$ as a test function in \cref{eq:PDE_xi_weaker_metrics}, and combining with the above, we obtain
  \begin{align*}
    &\int_{\Omega}\left(\xi_t-\int_\Omega \xi_t d\rho_t\right)^{2}\,dx  \leq C\int_{\Omega}\inner{\left(D^{2}\phi_t\right)^{-1}\nabla\xi_t}{\nabla\xi_t} f_t	\\
    & =-C\int_{\Omega}(f_1-f_0)\left(\xi_t-\int_\Omega \xi_t d\rho_t\right)+C\int_{\Omega}\frac{g_1(\nabla\phi_t)-g_0(\nabla\phi_t)}{g_{t}(\nabla\phi_t)}f_{t}\left(\xi_t-\int_\Omega \xi_t d\rho_t\right) \\
    & \leq C\int_{\Omega}(f_1-f_0)^{2}+C\int_{\Omega_{*}}\left(g_1-g_0\right)^{2}+\frac{1}{2}\int_{\Omega}\left(\xi_t-\int_\Omega \xi_t d\rho_t\right)^{2}
  \end{align*}
  which shows
  \begin{equation}\label{eqn: variance bound}
    \left\Vert \xi_t-\int_\Omega \xi_t d\rho_t\right\Vert _{L^{2}(\Omega)}\leq C\left(\left\Vert f_1-f_0\right\Vert _{L^{2}(\Omega)}+\left\Vert g_1-g_0\right\Vert _{L^{2}(\Omega_*)}\right).
  \end{equation}
  For a fixed $x\in \overline{\Omega}$ we calculate,
  \begin{align*}
      \int_0^1\left(\xi_t(x)-\int_\Omega \xi_t d\rho_t\right)dt
      &=\phi_1(x)-\phi_0(x)- \int_0^1\int_\Omega\partial_t\phi_t(y)e^{-\phi_t(y)}dydt\\
      &=\phi_1(x)-\phi_0(x)+ \int_\Omega\int_0^1\partial_t\left(e^{-\phi_t(y)}\right)dtdy\\
      &=\phi_1(x)-\phi_0(x)+ \int_\Omega\left(e^{-\phi_1(y)}-e^{-\phi_0(y)}\right)dy=\phi_1(x)-\phi_0(x).
  \end{align*}
  Now the stability bound is a consequence of Minkowski's integral inequality combined with~\eqref{eqn: variance bound}:
  \begin{equation*}
    \left\Vert \phi_{1}-\phi_{0}\right\Vert _{L^{2}(\Omega)}=\left\Vert \int_{0}^{1}\left(\xi_t-\int_\Omega \xi_t d\rho_t\right)\,dt\right\Vert _{L^{2}(\Omega)}\leq C\left(\left\Vert f_1-f_0\right\Vert _{L^{2}(\Omega)}+\left\Vert g_1-g_0\right\Vert _{L^{2}(\Omega_*)}\right).
  \end{equation*}
  This finishes the proof of \cref{thm:L2_L2_stab_potential} for smooth densities. 
  
  Now note that the regularity of $f_i$, $g_i$ was only used to justify the differentiability of $\phi_t$ with respect to $t$ and the bound obtained at the end is independent of such regularity. Therefore, an approximation argument allows us to extend the inequality to discontinuous densities: for general $f_i, g_i$, $i=0,1$ as in the hypotheses, let $\{f_{i, j}\}_{j\in \N}$ and $\{g_{i, j}\}_{j\in \N}$ be smooth densities with the same upper and lower bounds converging to $f_i$ and $g_i$ in $L^2$, with corresponding Brenier potentials $\{\phi_{i, j}\}_{j\in \N}$ which have average zero. Since $\nabla \phi_{i, j}(\Omega)\subset \Omega_*$, we can use Arzel{\`a}--Ascoli to obtain subsequences that converge uniformly on $\Omega$ to $\phi_i$, necessarily convex; note the Legendre transforms of $\phi_{i, j}$ also uniformly converge to the Legendre transform of $\phi_i$. In particular, we can take a limit in the Kantorovich dual problem (see \cite[Theorem 5.10]{villani2009optimal_old_new}) to see $\phi_i$ is also a Brenier potential from $f_i$ to $g_i$. Thus we can take a limit in the corresponding estimates for $\phi_{i, j}$ to finish the proof of~\eqref{eq:L2_L2}.

  Finally,~\eqref{eqn: L2 grad bound} follows immediately from~\eqref{eq:L2_L2} combined with \cite[Proposition 4.1]{MerigotDela2023StabOt}.
\end{proof}

\section{Proof of \cref{thm:unif_stab}}\label{sec:uniform_stability_potentials}
Theorem \ref{thm:unif_stab} is derived as a consequence of boundary regularity estimates for elliptic equations. First, we use \cite[Theorem 5.31]{lieberman2013oblique}, which we re-write in our setting as
\begin{lemma}\label{lem:moser_estimate}
  Let $\Omega \subset \mathbb R^d$ be a bounded Lipschitz domain and let $u$ be a weak solution of
  \begin{equation*}
    \begin{cases}
  -\div\left(A\nabla u\right)=f, & \textnormal{in }\Omega,\\
  \inner{A\nabla u}{\normal} =\psi, & \textnormal{on }\partial\Omega,\\
  \int_{\Omega}u=0,&
  \end{cases}
  \end{equation*}
  where $\Lambda^{-1}\Id\leq A\leq\Lambda\Id$, $f\in L^{q}(\Omega)$ and $\psi\in L^{q-1}(\partial \Omega)$ for some $q>\frac{d}{2}$. Then
  \begin{equation*}
    \left\Vert u\right\Vert _{L^{\infty}(\Omega)}\leq C
    \left(\Vert f\Vert _{L^{q}(\Omega)}+ \Vert \psi\Vert_{L^{q-1}(\partial\Omega)}\right),
  \end{equation*}
  where $C$ depends on $\Omega$, $\Lambda$, and $q$.
\end{lemma}

We will also make use of the Schauder estimate \cite[Theorem 5.54]{lieberman2013oblique}, which we write now in a simplified way that serves our purpose.
\begin{lemma}\label{lem:Schauder_estimate}
  Let $\Omega$ be a bounded $\mathcal C^{1,\alpha}$ domain, and let $u$ be a weak solution of
  \begin{equation*}
    \begin{cases}
      -\div\left(A\nabla u\right) = f, & \textnormal{in }\Omega \\ 
      \inner{A\nabla u}{\normal} = \psi,  & \textnormal{on }\partial\Omega
    \end{cases}
  \end{equation*}
  where $\Lambda^{-1}\Id\leq A\leq \Lambda \Id$, $A\in \mathcal C^{0,\alpha}(\overline \Omega)$, $f\in L^{\frac{d}{1-\alpha}}(\Omega)$ and $\psi\in \mathcal C^{0,\alpha}(\partial \Omega)$. Then
  \begin{equation*}
    \Vert u\Vert_{\mathcal C^{1,\alpha}(\overline \Omega)} \leq C\left(
    \Vert u \Vert_{L^\infty(\Omega)} + \Vert f\Vert_{L^{\frac{d}{1-\alpha}}(\Omega)}
    + \Vert \psi \Vert_{\mathcal C^{0,\alpha}(\partial \Omega)}
    \right),
  \end{equation*}
  where $C$ depends on $\Omega$, $\alpha$, $\Lambda$, and $[A]_{\mathcal C^{0,\alpha}(\overline \Omega)}$.
\end{lemma}

Now we are in position to prove Theorem \ref{thm:unif_stab}. For the rest of the section, $C$ will denote a constant which may change from line to line and that depends on $\Omega$, $\Omega_*$, $\alpha$, $\beta$, $a$, $\Vert f_i\Vert _{\mathcal{C}^{0,\beta}(\overline{\Omega})}$, and $\Vert g_i\Vert _{\mathcal{C}^{0,\beta}(\overline{\Omega}_{*})}$.

\begin{proof}[Proof of Theorem \ref{thm:unif_stab}]
  As before, we consider
  \begin{equation*}
    f_{t}=(1-t)f_0+tf_1,\qquad g_{t}=(1-t)g_0+tg_1
  \end{equation*}
  and interpolate between $\phi_0$, $\phi_1$ by letting $\phi_t$ be the Brenier potential from $f_{t}$ to $g_{t}$ with $\int_\Omega \phi_t=0$ for all $t\in [0, 1]$. Then Corollary \ref{cor:implicit_fun} shows that $\xi_{t}:=\partial_{t}\phi_{t}$ is a weak solution of
  \begin{equation*}
    \begin{cases}
    {\displaystyle -\div\left(f_{t}\left(D^{2}\phi_{t}\right)^{-1}\nabla\xi_{t}\right)=\frac{g_1(\nabla\phi_{t})-g_0(\nabla\phi_{t})}{g_{t}(\nabla\phi_{t})}f_{t}-(f_1-f_0),} & \textnormal{in }\Omega,\\
    \left\langle f_t \left(D^{2}\phi_{t}\right)^{-1}\nabla\xi_t,\normal\right\rangle =0, & \textnormal{on }\partial\Omega.
    \end{cases}
  \end{equation*}
  As done previously, we use Lemma \ref{lem: global C2 estimate} to obtain a uniform ellipticity estimate %
  \begin{equation*}
    C^{-1} \Id \leq f_t(D^2\phi_t)^{-1} \leq  C \Id,
  \end{equation*}
  then combining this with Lemma \ref{lem:prod_quotient_Holder} shows 
  \begin{equation*}
    \Vert f_{t}(D^{2}\phi_{t})^{-1}\Vert_{\mathcal{C}^{0,\alpha}(\overline{\Omega})}\leq C.
  \end{equation*}
  Therefore, we can apply Lemma \ref{lem:Schauder_estimate} followed by Lemma \ref{lem:moser_estimate} to obtain
  \begin{align*}
    \left\Vert \xi_t\right\Vert _{\mathcal{C}^{1,\alpha}(\overline \Omega)} &
    \leq C\left(\left\Vert \xi_t\right\Vert _{L^{\infty}(\Omega)}+\left\Vert \frac{g_1(\nabla\phi_{t})-g_0(\nabla\phi_{t})}{g_{t}(\nabla\phi_{t})}f_{t}-(f_1-f_0)\right\Vert _{L^{\frac{d}{1-\alpha}}(\Omega)}\right)	\\
    &\leq C\left\Vert \frac{g_1(\nabla\phi_{t})-g_0(\nabla\phi_{t})}{g_{t}(\nabla\phi_{t})}f_{t}-(f_1-f_0)\right\Vert _{L^{\frac{d}{1-\alpha}}(\Omega)}	\\
    & \le C\left(\left\Vert f_1-f_0\right\Vert _{L^{\frac{d}{1-\alpha}}(\Omega)}+\left\Vert g_1-g_0\right\Vert _{L^{\frac{d}{1-\alpha}}(\Omega_{*})}\right).
  \end{align*}
  By the above bound, we can justify exchanging integration and differentiation below to obtain
  \begin{align*}
    \left\Vert \phi_{1}-\phi_{0}\right\Vert _{\mathcal{C}^{1,\alpha}(\overline \Omega)}
    &=\left\Vert \int_{0}^{1}\xi_{t}\,dt\right\Vert _{\mathcal{C}^{1,\alpha}(\overline \Omega)}
    \leq \int_{0}^{1}\left\lVert \xi_{t}\right\rVert _{\mathcal{C}^{1,\alpha}(\overline \Omega)}\,dt\\
    &\leq C\left(\left\Vert f_1-f_0\right\Vert _{L^{\frac{d}{1-\alpha}}(\Omega)}+\left\Vert g_1-g_0\right\Vert _{L^{\frac{d}{1-\alpha}}(\Omega_{*})}\right).
    \qedhere
  \end{align*}
\end{proof}
Now we prove the statements from Remark \ref{rmk:sharpness} on the necessity of the lower bound of the target measures and the one dimensional stability estimate.
\begin{proof}[Proof of Remark \ref{rmk:sharpness}]
    Let us first show the one dimensional bound
    $$
    \left\Vert \phi_{1}'-\phi_{0}'\right\Vert _{L^{\infty}(I)}\leq\frac{1}{a}\left(\left\Vert f_1-f_0\right\Vert _{L^{1}(I)}+\left\Vert g_1-g_0\right\Vert _{L^{1}(J)}\right),$$
    which we stated in (\ref{eq:1D_Linfty_L1}). We use that in one dimension the optimal transport map can be simply expressed as
    $$
    \phi_{i}'=G_i^{-1}\circ F_i,
    $$
    where $F_i$, $G_i$ are the cumulative distribution functions of $f_i$ and $g_i$. From here the bound is an elementary computation. Indeed, for $x_1\leq x_2$ in $J$,
    \begin{align*}
        \lvert G_i(x_1)-G_i(x_2)\rvert
        &=\int_{x_1}^{x_2}g_i\geq a\lvert x_2-x_1\rvert,
    \end{align*}
    hence the $G_i^{-1}$ are $\frac{1}{a}$-Lipschitz and we have 
    \begin{align*}
        \left|G_1^{-1}\left(F_1(x)\right)-G_0^{-1}\left(F_0(x)\right)\right| &
        \leq\left|G_1^{-1}\left(F_1(x)\right)-G_1^{-1}\left(F_0(x)\right)\right|+\left|G_1^{-1}\left(F_0(x)\right)-G_0^{-1}\left(F_0(x)\right)\right|\\
        & \le\frac{1}{a}\left\Vert F_1-F_0\right\Vert _{L^{\infty}(I)}+\left\Vert G_1^{-1}-G_0^{-1}\right\Vert _{L^{\infty}([0,1])}\\
        & \leq\frac{1}{a}\left\Vert f_1-f_0\right\Vert _{L^{1}(I)}+\left\Vert G_1^{-1}-G_0^{-1}\right\Vert _{L^{\infty}([0,1])}.
    \end{align*}
    For the second term, we note that given $y\in[0,1]$
    $$
    \left|G_1^{-1}(y)-G_0^{-1}(y)\right|=\left|G_0^{-1}\left(G_0\left(G_1^{-1}(y)\right)\right)-G_0^{-1}(y)\right|\leq\frac{1}{a}\left|G_0\left(G_1^{-1}(y)\right)-y\right|.
    $$
    Therefore
    \begin{align*}
        \left\Vert G_1^{-1}-G_0^{-1}\right\Vert _{L^{\infty}([0,1])} & 
        \leq\frac{1}{a}\sup_{y\in[0,1]}\left|G_0\left(G_1^{-1}(y)\right)-y\right|\\
        & =\frac{1}{a}\sup_{x\in J}\left|G_0(x)-G_1(x)\right|\leq\frac{1}{a}\left\Vert g_1-g_0\right\Vert _{L^{1}(J)},
    \end{align*}
    which proves (\ref{eq:1D_Linfty_L1}). 
    
    Next, we show that, if we remove the lower bound hypothesis $g_i\geq a$ from Theorem \ref{thm:unif_stab}, there exists no pair of constants $0<\eta <1$ and $C>0$ such that
    $$
    \left\Vert \nabla\phi_{1}-\nabla\phi_{0}\right\Vert _{L^{\infty}(\Omega)}\leq C\left(\left\Vert f_0-f_1\right\Vert _{L^{\infty}(\Omega)}^\eta+\left\Vert g_0-g_1\right\Vert _{L^{\infty}(\Omega_{*})}^\eta\right).
    $$
    We prove this with a one dimensional example on $\Omega = \Omega_*=(0,1)$. We set $f\equiv 1$ and
    $$
    g_{c}(x)=
    \begin{cases}
        (p+1)c^{-p}(c-x)^{p}, & 0\leq x\leq c\\
        (p+1)(1-c)^{-p}(x-c)^{p}, & c\leq x\leq1
    \end{cases}, \qquad 0<c<1.
    $$
    In this case, the optimal transport map between $f$ and $g_c$ is given by $G_c^{-1}$, where $G_c$ is the cumulative function of $g_c$. We will consider $c$ in a small neighborhood of $\frac{1}{2}$ so that, for each fixed $p>0$, the functions $g_c$ have a uniform $\mathcal C^{0,\min(p,1)}(\Omega)$ bound. A simple computation shows that
    $$
    G_{c}^{-1}(y)=
    \begin{cases}
        c\left(1-\left(1-\frac{y}{c}\right)^{\frac{1}{p+1}}\right), & 0\leq y\leq c,\\
        c+(1-c)\left(\frac{y-c}{1-c}\right)^{\frac{1}{p+1}}, & c\leq y\leq1.
    \end{cases}
    $$
    Since the $g_c$ are defined so that $G_c(c)=c$, setting $c=\frac{1}{2(1-\varepsilon)}=\frac{1}{2}+\frac{\varepsilon}{2}+O(\varepsilon^{2})$ we have
    $$
    \left\Vert G_{c}^{-1}-G_{1/2}^{-1}\right\Vert _{L^{\infty}([0,1])}\geq\left|G_{c}^{-1}(1/2)-G_{1/2}^{-1}(1/2)\right|=\left|\frac{1-\varepsilon^{\frac{1}{p+1}}}{2(1-\varepsilon)}-\frac{1}{2}\right|\gtrsim\varepsilon^{\frac{1}{p+1}}.
    $$
    On the other hand, a simple computation shows that for each fixed $p>1$ it holds
    $$
    \left\Vert g_{c}-g_{1/2}\right\Vert _{L^{\infty}([0,1])}\lesssim \varepsilon.
    $$
    Therefore, for any $\frac{1}{p+1}<\eta<1$ we have
    $$
    \limsup_{c\to \frac{1}{2}}\frac{\left\Vert G_{c}^{-1}-G_{1/2}^{-1}\right\Vert _{L^{\infty}([0,1])}}{\left\Vert g_{c}-g_{1/2}\right\Vert _{L^{\infty}([0,1])}^{\eta}}=+\infty.\qedhere
    $$
\end{proof}

\section{Proof of \cref{thm:2var_d2}}\label{sec:second_variation}

As a preliminary step, we show that the $2$-Monge--Kantorovich distance is twice differentiable. In particular, the formula from \cref{thm:2var_d2} may be understood pointwise.

\begin{lemma}\label{lem:d2_twice_diff}
    Let $f,g,h,k$ as in Theorem \ref{thm:2var_d2} and $f_t=f(1+th)$, $g_t=g(1+tk)$. Then $\mk_2(f_t,g_t)^2$ is twice differentiable in $t$ for $|t|$ small enough.
\end{lemma}

\begin{proof}
    Recall by Kantorovich duality (see \cite[Theorem 5.10]{villani2009optimal_old_new}) it holds for each $t$,
    $$
    \frac{1}{2}\mk_2(f_t, g_t)^2=\sup\left\{ \int_\Omega\left(\frac{\lvert x\rvert^2}{2}-\phi(x)\right)f_t(x)\,dx+\int_{\Omega_*}\left(\frac{\lvert y\rvert^2}{2}-\psi(y)\right)g_t(y)\,dy\right\},
    $$
    where the supremum is taken over $\phi(x)$, $\psi(y)$ satisfying $\phi(x)+\psi(y)\leq \inner{x}{y}$. Then the pair $(\phi_t, \phi_t^*)$ is a maximizer in the problem above, where $\phi_t^*$ is the Legendre transform of $\phi_t$ defined by
    \begin{align*}
        \phi_t^*(y):=\sup_{x\in \Omega}(\inner{x}{y}-\phi_t(x)).
    \end{align*}
    Moreover, if $T_t=\nabla\phi_t(x)$ is the optimal map from $f_t$ to $g_t$, then under our conditions $T_t$ is invertible and it holds for all $y\in \Omega_*$ that
    \begin{align}\label{eqn: dual potential equality}
        \phi_t^*(y)=\inner{T_t^{-1}(y)}{y}-\phi_t(T_t^{-1}(y)).
    \end{align}
    Regarding differentiability in $t$, Corollary \ref{cor:implicit_fun} shows that $\nabla\phi_t$ is $\mathcal C^1$ with respect to $t$. We similarly have that $T_t^{-1}(y)$ is differentiable in $t$ for any $y\in \Omega_*$ and $\lvert t\rvert$ small enough, with $\partial_tT_t^{-1}(y)$ uniformly bounded for such $(t, y)$. Thus~\eqref{eqn: dual potential equality} yields that $\phi_t^*(y)$ is differentiable in $t$ for $y\in \Omega_*$, with uniformly bounded derivative, and we can justify differentiating under the integral to obtain
    \begin{align*}
        0
        &=\partial_s\vert_{s=0}\left[\int_\Omega\left(\frac{\lvert x\rvert^2}{2}-\phi_{t+s}(x)\right)f_t(x)\,dx+\int_{\Omega_*}\left(\frac{\lvert y\rvert^2}{2}-\phi_{t+s}^*(y)\right)g_t(y)\,dy\right]\\
        &=-\int_\Omega\partial_t\phi_t(x)f_t(x)\,dx-\int_{\Omega_*}\partial_t\phi_t^*(y)g_t(y)\,dy,
    \end{align*}
    where we have used that quantity being differentiated above achieves its maximum at $h=0$. Then we calculate, 
    \begin{align*}
        \frac{1}{2}\frac{d}{dt}\mk_2(f_{t},g_{t})^2 & 
        =\frac{d}{dt}\int_{\Omega}\left(\frac{|x|^2}{2}-\phi_{t}(x)\right)f_{t}(x)\,dx
        +\frac{d}{dt}\int_{\Omega_{*}}\left(\frac{|y|^2}{2}-\phi_{t}^*(y)\right)g_{t}(y)\,dy\\
        &=\int_{\Omega}\left(\frac{|x|^2}{2}-\phi_{t}(x)\right)\partial_{t}f_{t}(x)\,dx
        +\int_{\Omega_{*}}\left(\frac{|y|^2}{2}-\phi_{t}^*(y)\right)\partial_{t}g_{t}(y)\,dy\\
            & -\int_{\Omega}\partial_{t}\phi_{t}(x)f_{t}(x)\,dx-\int_{\Omega_{*}}\partial_{t}\phi_{t}^*(y)g_{t}(y)\,dy\\
            & =\int_{\Omega}\left(\frac{|x|^2}{2}-\phi_{t}(x)\right)h(x)f(x)\,dx
            +\int_{\Omega_{*}}\left(\frac{|y|^2}{2}-\phi_{t}^*(y)\right)k(y)g(y)\,dy.
    \end{align*}
    We can see the last expression above is also differentiable in $t$, proving the claim.
\end{proof}

\begin{proof}[Proof of \cref{thm:2var_d2}]
First we assume $f$, $g$, $h$, $k$ are all $\mathcal C^\infty$. Note that for $|t|<\delta$ small enough we have that $f_t = (1+th)f$, $g_t = (1+tk)g$ are bounded away from zero, so we can apply Theorem \ref{thm:implicit_fun} and conclude that for any $0<\beta <\alpha$ the mapping $t\mapsto \phi_t\in \mathcal C^{2,\beta}(\overline \Omega)$ is $\mathcal C^2$. This also implies that $\phi\in \mathcal C^{2}\big((-\delta,\delta)\times \overline \Omega\big)$ as a function of $t$, $x$ and that
\begin{equation*}
  t\mapsto \nabla \phi_t \in \mathcal C^{1,\beta}(\overline \Omega), \qquad
  t\mapsto D^2\phi_t\in \mathcal C^{0,\beta}(\overline \Omega;\mathbb R^{d\times d})
\end{equation*}
are twice differentiable as curves on Banach spaces. This justifies the application of the dominated convergence theorem in the subsequent computations. 

To obtain the exact expression in~\eqref{eq:second_var_thm}, we compute the discrete difference
\begin{align*}
  \frac{\mk_2(f_{t},g_{t})^2+\mk_2(f_{-t},g_{-t})^2-2\mk_2(f,g)^2}{t^{2}} &
  =\frac{1}{t^{2}}\int_{\Omega}\left|\nabla\phi_{t}-x\right|^{2}f_{t}\,dx	\\
  & +\frac{1}{t^{2}}\int_{\Omega}\left|\nabla\phi_{-t}-x\right|^{2}f_{-t}\,dx-\frac{2}{t^{2}}\int_{\Omega}\left|\nabla\phi-x\right|^{2}f\,dx\\
  & =\int_{\Omega}\underbrace{\frac{\left|\nabla\phi_{t}-x\right|^{2}+\left|\nabla\phi_{-t}-x\right|^{2}-2\left|\nabla\phi-x\right|^{2}}{t^{2}}}_{I}\,f\,dx\\
  & +\int_{\Omega}\frac{\left|\nabla\phi_{t}-x\right|^{2}-\left|\nabla\phi_{-t}-x\right|^{2}}{t}hf\,dx.
\end{align*}
After some algebraic manipulations we have
\begin{equation*}
  I=\inner{\frac{(\nabla\phi_{t}-\nabla \phi)+(\nabla \phi-\nabla\phi_{-t})}{t}}{\frac{\nabla\phi_{t}-\nabla\phi}{t}}+\inner{\nabla\phi_{-t}+\nabla\phi-2x}{\frac{\nabla\phi_{t}+\nabla\phi_{-t}-2\nabla\phi}{t^{2}}},
\end{equation*}
and we similarly obtain
\begin{equation*}
  \frac{\left|\nabla\phi_{t}-x\right|^{2}-\left|\nabla\phi_{-t}-x\right|^{2}}{t}
  =\inner{\nabla\phi_{t}+\nabla\phi_{-t}-2x}{\frac{(\nabla\phi_{t}-\nabla \phi)+(\nabla \phi-\nabla\phi_{-t})}{t}}.
\end{equation*}
Therefore, letting $t\to 0$ and denoting $\xi=\partial_{t}|_{t=0}\phi_t$, we deduce
\begin{align} \nonumber
  \left.\frac{d^{2}}{dt^{2}}\right|_{t=0}\mk_2(f_{t},g_{t})^2 &
  = 4\int_{\Omega}\inner{\nabla\phi(x)-x}{\nabla\xi(x)}h(x)f(x)\,dx
  + 2\int_{\Omega}\left|\nabla\xi\left(x\right)\right|^{2}f(x)\,dx	\\
  & +2\int_{\Omega}\inner{\nabla\phi(x)-x}{\left.\frac{d^{2}}{dt^{2}}\right|_{t=0}\nabla\phi_{t}(x)}f(x)\,dx
  \label{eq:missing_term_2var_OT}
\end{align}
Next, we compute
\begin{equation*}
  \int_{\Omega}\inner{\nabla\eta\left(\nabla\phi(x)\right)}{\left.\frac{d^{2}}{dt^{2}}\right|_{t=0}\nabla\phi_{t}(x)}f(x)\,dx,
\end{equation*}
where $\eta\in \mathcal C^2(\overline \Omega_*)$ is a general test function. To do this we note that
\begin{align*}
  0 & =\int_{\Omega_*}\frac{g_{t}(y)+g_{-t}(y)-2g(y)}{t^{2}}\eta(y)\,dy	\\
  & =\frac{1}{t^{2}}\int_{\Omega}\eta\left(\nabla\phi_{t}(x)\right)f_{t}(x)\,dx
  +\frac{1}{t^{2}}\int_{\Omega}\eta\left(\nabla\phi_{-t}(x)\right)f_{-t}(x)\,dx
  -\frac{2}{t^{2}}\int_{\Omega}\eta\left(\nabla\phi(x)\right)f(x)\,dx\\
  & =\int_{\Omega}\frac{\eta\left(\nabla\phi_{t}(x)\right)+\eta\left(\nabla\phi_{-t}(x)\right)-2\eta\left(\nabla\phi(x)\right)}{t^{2}}f(x)\,dx
  \\
  &\qquad\qquad+\int_{\Omega}\frac{\big(\eta\left(\nabla\phi_{t}(x)\right)-\eta\left(\nabla\phi(x)\right)\big)
  +\big(\eta\left(\nabla\phi(x)\right)-\eta\left(\nabla\phi_{-t}(x)\right)\big)}{t}h(x)f(x)\,dx.
\end{align*}
Letting $t\to0$ yields
\begin{align*}
  0 &	=\int_{\Omega}\left.\frac{d^{2}}{dt^{2}}\right|_{t=0}\left[\eta\left(\nabla\phi_{t}(x)\right)\right]f(x)\,dx
  +2\int_{\Omega}\left.\frac{d}{dt}\right|_{t=0}\left[\eta\left(\nabla\phi_{t}(x)\right)\right]h(x)f(x)\,dx\\
  & =\int_{\Omega}\left\langle D^{2}\eta\left(\nabla\phi(x)\right)\nabla\xi\left(x\right),\nabla\xi\left(x\right)\right\rangle f(x)\,dx
  +\int_{\Omega}\inner{\nabla\eta\left(\nabla\phi(x)\right)}{\left.\frac{d^{2}}{dt^{2}}\right|_{t=0}\nabla\phi_{t}(x)}f(x)\,dx\\
  & \,+2\int_{\Omega}\inner{\nabla\eta\left(\nabla\phi(x)\right)}{\nabla\xi\left(x\right)}h(x)f(x)\,dx,
\end{align*}
which shows that
\begin{align*}
  \int_{\Omega}\inner{\nabla\eta\left(\nabla\phi(x)\right)}{\left.\frac{d^{2}}{dt^{2}}\right|_{t=0}\nabla\phi_{t}(x)}f(x)\,dx &
  =-\int_{\Omega}\left\langle D^{2}\eta\left(\nabla\phi(x)\right)\nabla\xi\left(x\right),\nabla\xi\left(x\right)\right\rangle f(x)\,dx	\\
  & \,-2\int_{\Omega}\inner{\nabla\eta\left(\nabla\phi(x)\right)}{\nabla\xi\left(x\right)}h(x)f(x)\,dx.
\end{align*}
Now we can set ${\eta(y)= \frac{1}{2}|y|^2-\psi(y)}$ where $\psi$ is a conjugate Brenier potential from $g$ to $f$ satisfying ${\nabla \psi = (\nabla \phi)^{-1}}$ to obtain
\begin{align*}
  \int_{\Omega}\inner{\nabla\phi(x)-x}{\left.\frac{d^{2}}{dt^{2}}\right|_{t=0}\nabla\phi_{t}(x)}f(x)\,dx &
  =-\int_{\Omega}\left\langle \left(\Id-D^{2}\psi\left(\nabla\phi(x)\right)\right)\nabla\xi(x),\nabla\xi(x)\right\rangle f(x)\,dx	\\
  & \,-2\int_{\Omega}\inner{\nabla\phi(x)-\nabla\psi\left(\nabla\phi(x)\right)}{\nabla\xi(x)}h(x)f(x)\,dx\\
  & =-\int_{\Omega}\left\langle \left(\Id-\big(D^{2}\phi(x)\big)^{-1}\right)\nabla\xi(x),\nabla\xi(x)\right\rangle f(x)\,dx\\
  & \,-2\int_{\Omega}\inner{\nabla\phi(x)-x}{\nabla\xi(x)}h(x)f(x)\,dx.
\end{align*}
Plugging this into (\ref{eq:missing_term_2var_OT}) yields 
\begin{equation*}
  \left.\frac{d^{2}}{dt^{2}}\right|_{t=0}\mk_2(f_{t},g_{t})^2
  =2\int_{\Omega}\inner{\left(D^{2}\phi(x)\right)^{-1}\nabla\xi(x)}{\nabla\xi(x)} f(x)\,dx,
\end{equation*}
which proves Theorem \ref{thm:2var_d2} for smooth densities.

Next, we prove Theorem \ref{thm:2var_d2} assuming only $f$, $h\in \mathcal C^{0,\alpha}(\overline \Omega)$, $g$, $k\in \mathcal C^{0,\alpha}(\overline \Omega_*)$ through an approximation argument. There exist linear and bounded extension operators
$$
\mathcal E:\mathcal C^{0,\alpha}(\overline \Omega)\longrightarrow C^{0,\alpha}(\mathbb R^d), \qquad
\mathcal E_*:\mathcal C^{0,\alpha}(\overline \Omega_*)\longrightarrow C^{0,\alpha}(\mathbb R^d),
$$
constructed as in $\mathcal{E}_0$ in \cite[Section 2.2 (8)]{stein1970singular}. 
Therefore, without loss of generality, we may assume that $f,g,h,k\in\mathcal{C}^{0,\alpha}(\mathbb{R}^{d})$ with bounds of the form ${\left\Vert f\right\Vert _{\mathcal{C}^{0,\alpha}(\mathbb{R}^{d})}\leq c_d\left\Vert f\right\Vert _{\mathcal{C}^{0,\alpha}(\overline \Omega)}}$. Then given an approximation kernel $m^\varepsilon\in \mathcal C^\infty_c(\mathbb{R}^d)$, we set
\begin{equation*}
  f^{\varepsilon}=c_{\varepsilon,f}(f*m^{\varepsilon}),\quad g^{\varepsilon}=c_{\varepsilon,g}(g*m^{\varepsilon}),\quad h^{\varepsilon}=h*m^{\varepsilon}-c_{\varepsilon,h},\quad k^{\varepsilon}=k*m^{\varepsilon}-c_{\varepsilon,k},
\end{equation*}
where the constants are chosen so that $f^{\varepsilon},g^{\varepsilon}$ are probability densities on $\Omega,\Omega_{*}$ and
\begin{equation*}
  \int_{\Omega}h^{\varepsilon}f^{\varepsilon}=\int_{\Omega_{*}}k^{\varepsilon}g^{\varepsilon}=0.
\end{equation*}
Choosing $\varepsilon>0$ small enough it is easy to see that $f^{\varepsilon},g^{\varepsilon}$ will be bounded from below by $a/2$, in which case we can define the curve of densities 
\begin{equation*}
  f_{t}^{\varepsilon}:=(1+th^{\varepsilon})f^{\varepsilon},\qquad g_{t}^{\varepsilon}:=(1+tk^{\varepsilon})g^{\varepsilon}
\end{equation*}
for $\left|t\right|<\delta$  sufficiently small. Let $\phi^\varepsilon_t$ be the Brenier potential from $f^\varepsilon_t$ to $g^\varepsilon_t$ with average zero and $\xi^\varepsilon_t:=\partial_t\phi^\varepsilon_t$. Since the densities are smooth, we can apply the previous computations to $\phi_t$ replaced by $\phi^\varepsilon_{t+u}$ for $\lvert u\rvert$ sufficiently small to obtain
\begin{equation*}
  \left.\frac{d^{2}}{dt^{2}}\right|_{t=u}\mk_2(f_{t}^{\varepsilon},g_{t}^{\varepsilon})^2=2\int_{\Omega}\left\langle \left(D^{2}\phi_{u}^{\varepsilon}\right)^{-1}\nabla\xi_{u}^{\varepsilon},\nabla\xi_{u}^{\varepsilon}\right\rangle f_{u}^{\varepsilon}.
\end{equation*}
 Now we let $\varepsilon\to 0$ after re-writing the previous expression in integral form. In general, if $\varphi:[-t,t]\longrightarrow\mathbb{R}$ is a $\mathcal{C}^{2}$ function, the second order quotient can be re-written as
\begin{equation*}
  \frac{\varphi(t)+\varphi(-t)-2\varphi(0)}{t^{2}}
  =2\int_{0}^{1}\int_{0}^{1}s\varphi''\left((2\tau-1)st\right)\,d\tau\,ds.
\end{equation*}
Denoting $u_{t}(\tau,s):=(2\tau-1)st\in[-t,t]$, we see $\lvert u_t(\tau, s)\rvert$ can be made small uniformly in $\tau$, $s\in [0, 1]$ by taking $\lvert t\rvert$ small, thus applying the above integral identity to $t\mapsto \mk_2(f_{t}^{\varepsilon},g_{t}^{\varepsilon})^2$ we obtain
\begin{align*}
  &\frac{\mk_2(f_{t}^{\varepsilon},g_{t}^{\varepsilon})^2+\mk_2(f_{-t}^{\varepsilon},g_{-t}^{\varepsilon})^2-2\mk_2(f_0^{\varepsilon},g_0^{\varepsilon})^2}{t^{2}}\\
  &=4\int_{0}^{1}\int_{0}^{1}\int_{\Omega}s\left\langle \left(D^{2}\phi_{u_{t}(\tau,s)}^{\varepsilon}\right)^{-1}\nabla\xi_{u_{t}(\tau,s)}^{\varepsilon},\nabla\xi_{u_{t}(\tau,s)}^{\varepsilon}\right\rangle f_{u_{t}(\tau,s)}^{\varepsilon}\,d\tau\,ds.
\end{align*}
In order to pass to the limit we obtain estimates uniform in $\varepsilon$ and $t$. First, Lemma \ref{lem: global C2 estimate} implies that for $\beta<\alpha$,
\begin{equation*}
  \left\Vert \phi_{t}^{\varepsilon}\right\Vert _{\mathcal{C}^{2,\beta}(\overline{\Omega})}\leq C,
\end{equation*}
where $C$ depends on $\left\Vert f_{t}^{\varepsilon}\right\Vert _{\mathcal{C}^{0,\alpha}(\overline{\Omega})}$, $\left\Vert g_{t}^{\varepsilon}\right\Vert _{\mathcal{C}^{0,\alpha}(\overline{\Omega}_{*})}$ , and the lower bounds of $f^{\varepsilon}$ and $g^{\varepsilon}$. Restricting the interval for $t$ if necessary and choosing $\varepsilon$ small enough, all these quantities are controlled by the bounds of $f_{t}$ and $g_{t}$. Consequently, we obtain
\begin{equation*}
  \left\Vert \phi_{t}^{\varepsilon}\right\Vert _{\mathcal{C}^{2,\beta}(\overline{\Omega})}\leq C,
  \qquad \Lambda^{-1}\Id\leq D^{2}\phi_{t}^{\varepsilon}\leq\Lambda\Id
\end{equation*}
with constants independent of $t$, $\varepsilon$. Then classical elliptic regularity estimates (see Lemmas \ref{lem:moser_estimate} and \ref{lem:Schauder_estimate}) give uniform bounds on $\Vert \xi_t^\varepsilon\Vert_{\mathcal C^{1,\beta}(\overline \Omega)}$ which allows us to take $\varepsilon\to 0$ in the weak formulation of the PDE satisfied by $\xi_t^\varepsilon$ using the Arzelà--Ascoli Theorem. In summary, as $\varepsilon\to0$ we have the following convergence, uniformly for $\lvert u\rvert$ small:
\begin{equation*}
  \left(D^{2}\phi_{u}^{\varepsilon}\right)^{-1}\to\left(D^{2}\phi_{u}\right)^{-1},\qquad f_{u}^{\varepsilon}\to f_{u},\qquad\nabla\xi_{u}^{\varepsilon}\to\nabla\xi_{u}.
\end{equation*}
Therefore, letting $\varepsilon\to0$ and using the dominated convergence theorem gives
\begin{align*}
  &\frac{\mk_2(f_{t},g_{t})^2+\mk_2(f_{-t},g_{-t})^2-2\mk_2(f_0,g_0)^2}{t^{2}} \\
  &=
  4\int_{0}^{1}\int_{0}^{1}\int_{\Omega}s\left\langle \left(D^{2}\phi_{u_{t}(\tau,s)}\right)^{-1}\nabla\xi_{u_{t}(\tau,s)},\nabla\xi_{u_{t}(\tau,s)}\right\rangle f_{u_{t}(\tau,s)}\,d\tau\,ds.
\end{align*}
Given the uniform bounds on $\left\Vert \phi_{t}\right\Vert _{\mathcal{C}^{2,\beta}(\overline{\Omega})}$, $\left\Vert \xi_{t}\right\Vert _{\mathcal{C}^{1,\beta}(\overline{\Omega})}$, a simple compactness argument shows that $D^{2}\phi_{t}$, $\nabla\xi_{t}$ are continuous in $t$, so we can let $t\to0$ and the dominated convergence theorem gives
\begin{align*}
  \left.\frac{d^{2}}{dt^{2}}\right|_{t=0}\left[\mk_2(f_{t},g_{t})^2\right] 
  =4\int_0^1sds \int_{\Omega}\left\langle \left(D^{2}\phi\right)^{-1}\nabla\xi,\nabla\xi\right\rangle f\,dx
  =2\int_{\Omega}\left\langle \left(D^{2}\phi\right)^{-1}\nabla\xi,\nabla\xi\right\rangle f\,dx,
\end{align*}
which finishes the proof of Theorem \ref{thm:2var_d2}.
\end{proof}
\section{Proof of \cref{thm:implicit_fun}}\label{sec:time_reg_ifc}
Under our hypotheses for Theorem~\ref{thm:implicit_fun}, by \cite[Theorem 1.1]{chenLiuWang2021MAglobalReg} and bootstrapping standard elliptic regularity theory, Brenier potentials from $f_t$ to $g_t$ belong to $\mathcal{C}^{3, \alpha}(\overline\Omega)$ and satisfy equation~\eqref{eqn: MA eqn}. We will apply an implicit function theorem to show regularity in the parameter $t$ of Brenier potentials with average zero; the first step is to write \eqref{eqn: MA eqn} in a way that allows for linearization. Again we let $\omega_*$ be a convex defining function for $\Omega_*$ as in~\eqref{eqn: MA bdry cond}, in particular $\nabla\omega_*(y)=\normal_*(y)$ is the outer normal of $\Omega_*$. Now define the mapping
\begin{align*}
  \Gamma: & \,[0,1]\times\mathcal{C}_{u}^{2,\beta}(\overline{\Omega})\longrightarrow\mathcal{C}^{0,\beta}(\overline{\Omega})\times\mathcal{C}^{1,\beta}\left(\partial\Omega\right)	\\
  \Gamma(t,\phi) & =\left(\Gamma_{t}^{(1)}(\phi),\Gamma_{t}^{(2)}(\phi)\right):=\left(\log\det D^{2}\phi+\log g_{t}\left(\nabla\phi\right)-\log f_{t},\omega_{*}\left(\nabla\phi\right)\right),
\end{align*}
where $0<\beta <\alpha$ and $\mathcal{C}_{u}^{2,\beta}(\overline{\Omega})$ denotes the subset of $\mathcal{C}^{2,\beta}(\overline{\Omega})$ consisting of uniformly convex functions.
\begin{remark}
    Note that we have set the Hölder regularity of the densities to be $\mathcal C^{0,\alpha}$, but we consider $\Gamma$ defined between spaces defined for $0<\beta<\alpha$. This is a technical detail that allows us to show that $\nabla\phi\in \mathcal C^{1,\beta}(\overline \Omega)\mapsto g_t(\nabla \phi)\in \mathcal C^{0,\beta}(\overline \Omega)$ is differentiable as a map between Banach spaces using Lemma \ref{lem:reg_of_comp2}.
\end{remark}
First, we show regularity of $\Gamma$.

\begin{lemma}\label{lem:Gamma_Cm}
  If $\phi\in \mathcal C^{2,\beta}_{u}(\overline \Omega)$, then $\Gamma(t, \phi)\in \mathcal{C}^{0,\beta}(\overline{\Omega})\times\mathcal{C}^{1,\beta}\left(\partial\Omega\right)$. Moreover, under the hypotheses of Theorem \ref{thm:implicit_fun}, we have that $\Gamma$ is $m$ times continuously Frechét differentiable as a map $[0,1]\times\mathcal{C}_{u}^{2,\beta}(\overline{\Omega})\longrightarrow\mathcal{C}^{0,\beta}(\overline{\Omega})\times\mathcal{C}^{1,\beta}\left(\partial\Omega\right)$ with differential with respect to $\phi$
  \begin{equation*}
    (D_{\phi}\Gamma)(t, \phi)\xi=\left(\textnormal{tr}\left[(D^{2}\phi)^{-1}D^{2}\xi\right]+\frac{\inner{\nabla g_t(\nabla\phi)}{\nabla\xi}}{g_t(\nabla\phi)},\inner{\normal_*(\nabla\phi)}{\nabla\xi}\right).
  \end{equation*}
\end{lemma}
\begin{proof}
    The formula for $D_\phi\Gamma$ is well known and can be easily obtained by setting $\phi_s = \phi+s\xi$ and differentiating with respect to $s$. Thus it is sufficient to show $\Gamma$ is $m$ times continuously differentiable. 
    
    We first show the domain and co-domain of $\Gamma$ are as claimed. 
    Given a fixed $\phi\in \mathcal C^{2,\beta}_{u}(\overline \Omega)$, there is a constant $0<\Lambda$ such that $\Lambda^{-1} < D^2\phi(x)< \Lambda$ for all $x\in \overline \Omega$. Also note that the function $\log \det$ belongs to $\mathcal C^\infty(\mathcal{M}_{\Lambda})$ with all derivatives bounded when restricted to the convex set
  \begin{equation*}
    \mathcal{M}_{\Lambda}:= \left\{ A\in\mathbb{R}^{d\times d}\mid \Lambda^{-1}\,\Id<A<\Lambda\,\Id\right\},
  \end{equation*}
  therefore $\log \det D^2\phi\in \mathcal C^{0,\beta}(\overline \Omega)$ as a consequence of Lemma \ref{lem:reg_of_comp1}. Similarly, using that $f_t$, $g_t$ are bounded from above and below, Lemma \ref{lem:reg_of_comp1} shows that $\log g_t(\nabla \phi)\in \mathcal C^{1,\beta}(\overline \Omega)$, $\log f_t\in \mathcal C^{m,\beta}(\overline \Omega)$ and $\omega_*(\nabla\phi)\in\mathcal{C}^{1,\beta}(\partial\Omega)$.

  Let us now show regularity of $\Gamma$. In order to verify the hypotheses of Lemma \ref{lem:reg_of_comp2}, we consider the linear and bounded extension operator
    $$
    \mathcal{E}_0:\mathcal C^{m,\alpha}(\overline \Omega_*)\longrightarrow C^{m,\alpha}(\mathbb R^d).
    $$
    as constructed in \cite[Section 2.2 (8)]{stein1970singular}. We can then view $\Gamma=\mathcal{G}\circ \mathcal{F}$ as a composition of two mappings defined by
  \begin{align*}
      \mathcal{F}: I\times\mathcal{C}_{u}^{2,\beta}(\overline{\Omega}) 
      & \to\mathcal{C}^{1,\beta}(\overline{\Omega};\mathbb{R}^{d})
      \times\mathcal{C}^{0,\beta}(\overline{\Omega};\mathbb{R}^{d\times d}_{\textnormal{sym}})\times
      \mathcal{C}^{1,\alpha}(\overline{\Omega})\times
      \mathcal{C}^{m,\alpha}(\overline {\Omega_*^\delta})\notag\\
      \mathcal{F}(t,\phi) :&=(\nabla\phi,D^{2}\phi,f_{t},\mathcal{E}_0 g_{t}|_{\Omega_*^\delta}),
  \end{align*}
  where $\mathbb{R}^{d\times d}_{\textnormal{sym}}$ denotes the set of symmetric matrices, and $\Omega_*^\delta$ is a convex neighborhood of $\overline\Omega_*$ small enough to guarantee that $\mathcal{E}_0g_t > \frac{a}{2}$ on $\Omega_*^\delta$, and
  \begin{align*}
      &\mathcal{G}:\mathcal U
      \times\mathcal{C}_{>0}^{0,\beta}(\overline{\Omega};\mathbb{R}^{d\times d}_{\textnormal{sym}})
      \times\mathcal{C}_{>0}^{1,\alpha}(\overline{\Omega})
      \times\mathcal{C}_{>0}^{m,\alpha}(\overline {\Omega_*^\delta}) 
      \to\mathcal{C}^{0,\beta}(\overline{\Omega})\times\mathcal{C}^{1,\beta}(\partial\Omega)\\
      &\mathcal{G}(p,H,f,g) :=\left(\log\det H+\log(g\circ p)-\log f,\omega_{*}(p)\vert_{\partial\Omega}\right),
  \end{align*}
  where $\mathcal{C}_{>0}^{m,\alpha}(\overline {\Omega_*^\delta})$ is the subset of 
    $\mathcal{C}^{m,\alpha}(\overline {\Omega_*^\delta})$ of functions bounded away from $0$, 
    $\mathcal{C}_{>0}^{0,\beta}(\overline{\Omega};\mathbb{R}^{d\times d}_{\textnormal{sym}})$ is the subset of 
    $\mathcal{C}^{0,\beta} (\overline{\Omega};\mathbb{R}^{d\times d}_{\textnormal{sym}})$ which are positive definite uniformly over $\overline{\Omega}$, and
    $\mathcal U$ is the open subset of $p\in \mathcal{C}^{1,\beta}(\overline{\Omega};\mathbb{R}^{d})$ with $p(\overline \Omega)\subset \Omega_*^{\delta/2}$.
  
    It is easy to see that $\mathcal F$ is $m$ times differentiable, since $t\mapsto f_t,g_t$ are differentiable by hypothesis and the maps $\phi\mapsto (\nabla \phi, D^2\phi)$, $g\mapsto \mathcal{E}_0g|_{\Omega_*^\delta}$ are linear (hence infinitely Frechét differentiable). Note that $\mathcal{F}$ takes values in the set
    $$
    \mathcal U
      \times\mathcal{C}_{>0}^{0,\beta}(\overline{\Omega};\mathbb{R}^{d\times d}_{\textnormal{sym}})\times
      \mathcal{C}_{>0}^{1,\alpha}(\overline{\Omega})\times
      \mathcal{C}_{>0}^{m,\alpha}(\overline {\Omega_*^\delta}),
    $$
    therefore, by Lemma~\ref{lem:reg_of_comp1}, it is sufficient to show that $\mathcal{G}$ 
  is $m$ times differentiable, which we can do term by term.

  To this end, we now show that the mapping
  \begin{equation}\label{eq:log_map}
      (p,g)\in\left[\mathcal{U}\subset\mathcal{C}^{1,\beta}(\overline{\Omega};\mathbb{R}^{d})\right]\times\mathcal{C}_{>0}^{m,\alpha}(\overline{\Omega_{*}^{\delta}})
  \mapsto\log(g\circ p)\in\mathcal{C}^{0,\beta}(\overline{\Omega})
  \end{equation}
  is $m$ times Frechét differentiable. To do so, we view it as the composition $\mathcal L\circ \mathcal H$, where
  \begin{align*}
  \mathcal{H}&:\left[\mathcal{U}\subset\mathcal{C}^{1,\beta}(\overline{\Omega};\mathbb{R}^{d})\right]\times\mathcal{C}_{>0}^{m,\alpha}(\overline{\Omega_{*}^{\delta}})
  \longrightarrow\mathcal{C}_{>0}^{0,\beta}(\overline{\Omega}),
  \qquad \mathcal{H}(p,g)=g\circ p,\\
  \mathcal{L}&: \mathcal{C}_{>0}^{0,\beta}(\overline{\Omega})\longrightarrow\mathcal{C}^{0,\beta}(\overline{\Omega}),
  \qquad \mathcal{L}(h)=\log h.
  \end{align*}
  Note that $\mathcal H$ is Frechét differentiable due to Lemma \ref{lem:reg_of_comp2} statement \eqref{item:comp_reg} applied with $E=F=\mathbb R^d$, $G=\mathbb R$, $U=\Omega$, $V=\Omega_*^\delta$, $r=1+\beta$, $t=\beta$, $s=m+\alpha$. As for $\mathcal L$, we note that we can decompose
  $$
  \mathcal C_{>0}^{0,\beta}(\overline \Omega) = \cup_{\varepsilon>0}\mathcal V_\varepsilon,\qquad
  \mathcal V_\varepsilon:=\{f\in \mathcal C^{0,\beta}(\overline \Omega)\mid \varepsilon<f<\varepsilon^{-1}\},
  $$
  and it is enough to show that $\mathcal L$ is $m$ times Frechét differentiable on each open set $\mathcal V_\varepsilon$. For $\varepsilon>0$ fixed, we may restrict the logarithm to a sub-interval of $(0,\infty)$ where all its derivatives are bounded. Then Lemma \ref{lem:reg_of_comp2}, statement \eqref{item:g*_comp} with $E=\mathbb R^d$, $F=G=\mathbb R$, $U=\Omega$, $V=(\frac{\varepsilon}{2},\frac{2}{\varepsilon})$, $\mathcal U$ replaced by $\mathcal V_\varepsilon$, $r=t=\beta$, and $s=\infty$ shows that $\mathcal L$ is infinitely Frechét differentiable. All in all, this shows that \eqref{eq:log_map} is $m$ times Frechét differentiable.

  The terms $p\mapsto \omega_*(p)$, $H\mapsto \log \det H$ from $\mathcal G$ are treated analogously.
\end{proof}
By Remark~\ref{rmk:ellipticity_D2phi}, for a fixed $t$ the first coordinate of $D_\phi \Gamma(t, \phi_t)$ is a uniformly elliptic operator. Moreover, repeated applications of Lemmas \ref{lem:reg_of_comp1} and \ref{lem:prod_quotient_Holder} show that $D_\phi\Gamma(t, \phi_t)$ is a bounded linear operator.
To apply the implicit function theorem we need to show that, given a fixed $t$, the equation $D_\phi \Gamma(t, \phi_t)\xi=(p,q)$ has exactly one solution for each $p\in \mathcal C^{0,\beta}(\overline \Omega)$, ${q\in \mathcal C^{1,\beta}(\partial \Omega)}$. We have the following simplification:

\begin{lemma}\label{lem:magic_lemma}
  Given a fixed $t$ and $p\in \mathcal C^{0,\beta}(\overline \Omega)$, $q\in \mathcal C^{1,\beta}(\partial \Omega)$, the linearized Monge--Amp{\`e}re equation $D_\phi \Gamma(t, \phi_t)\xi=(p,q)$ is equivalent to the boundary value problem
  \begin{equation*}
    \begin{cases}
      \displaystyle{\div\left[f_t(D^{2}\phi_t)^{-1}\nabla\xi\right]=pf_t}, & \textnormal{in }\Omega, \\ 
      \displaystyle{\inner{(D^{2}\phi_t)^{-1}\nabla\xi}{\normal} =\left|(D^{2}\phi_t)^{-1}\normal\right|q}, & \textnormal{on }\partial\Omega.
    \end{cases}
  \end{equation*}
\end{lemma}
\begin{proof}
  According to Lemma \ref{lem:Gamma_Cm}, $D_\phi \Gamma(t, \phi_t)\xi=(p,q)$ corresponds to the problem
  \begin{equation*}
    \begin{cases}
      \displaystyle{\textnormal{tr}\left[(D^{2}\phi_t)^{-1}D^{2}\xi\right]+\frac{\inner{\nabla g_t(\nabla\phi_t)}{\nabla\xi}}{g_t(\nabla\phi_t)} = p},
      & \textnormal{in }\Omega, \\ 
      \inner{\normal_*(\nabla\phi_t)}{\nabla\xi} = q, & \textnormal{on }\partial\Omega.
    \end{cases}
  \end{equation*}
  Let us first re-write the boundary condition. Consider the function $H(x):=\omega_{*}\left(\nabla\phi_{t}(x)\right)$, which verifies $H(x)<0$ for $x\in\Omega$ and $H\equiv0$ on $\partial\Omega$. Then, if we fix $x_{0}\in\partial\Omega$, for any direction $\tau$  tangential to $\partial\Omega$  at $x_{0}$ we have 
  \begin{equation*}
    0=\partial_{\tau}H(x_{0})=\inner{D^{2}\phi_{t}(x_{0})\normal_{*}\left(\nabla\phi_{t}(x_{0})\right)}{\tau}.
  \end{equation*}
  As a consequence, we see that 
  \begin{equation*}
    D^{2}\phi_{t}(x_{0})\normal_{*}\left(\nabla\phi_{t}(x_{0})\right)=c\normal(x_{0})
  \end{equation*}
  for some $c\in\mathbb{R}$, $c\neq 0$ depending on $x_{0}$. This implies
  \begin{equation*}
    \normal_{*}\left(\nabla\phi_{t}(x)\right)=\frac{\left(D^{2}\phi_{t}(x)\right)^{-1}\normal(x)}{\left|\left(D^{2}\phi_{t}(x)\right)^{-1}\normal(x)\right|},
  \end{equation*}
  which shows the equivalence between the boundary conditions. Next, observe that
  \begin{equation}\label{eq:magic_eq}
    f_t\,\left(\textnormal{tr}\left[(D^{2}\phi_t)^{-1}D^{2}\xi\right]+\frac{\inner{\nabla g_t(\nabla\phi_t)}{\nabla\xi}}{g_t(\nabla\phi_t)}\right)=\div\left[f_t(D^{2}\phi_t)^{-1}\nabla\xi\right],
  \end{equation}
  which is shown as a simple calculus exercise after differentiating the Monge--Amp{\`e}re equation in the $x$ variable. Since $\phi_t\in \mathcal C^3(\overline \Omega)$, there are no regularity issues.
\end{proof}
Now we are ready to study the uniqueness of the linearized equation.

\begin{lemma}
  Consider a fixed $t$ and $p\in \mathcal C^{0,\beta}(\overline \Omega)$, $q\in \mathcal C^{1,\beta}(\partial \Omega)$. Then the equation $D_\phi \Gamma(t, \phi_t)\xi=(p,q)$ has a solution if and only if
  \begin{equation}\label{eq:compat_cond}
    \int_{\Omega}pf_t=\int_{\partial\Omega}\left|(D^{2}\phi_t)^{-1}\normal\right|qf_t.
  \end{equation}
  Moreover, this solution is unique up to additive constants.
\end{lemma}
\begin{proof}
  The previous Lemma \ref{lem:magic_lemma} tells us that $D_\phi \Gamma(t,\phi_t)\xi=(p,q)$ corresponds to the problem 
  \begin{equation*}
    \begin{cases}
      \displaystyle{\div\left[f_t(D^{2}\phi_t)^{-1}\nabla\xi\right]=pf_t}, & \textnormal{in }\Omega \\ 
      \displaystyle{\left\langle (D^{2}\phi_t)^{-1}\nabla\xi,\normal\right\rangle =\left|(D^{2}\phi_t)^{-1}\normal\right|q}, & \textnormal{on }\partial\Omega.
    \end{cases}
  \end{equation*}
  It is easy to see that (\ref{eq:compat_cond}) is a necessary condition, as 
  \begin{align*}
    \int_{\Omega}pf_{t}\,dx & 
    =\int_{\Omega}\div\left[f_{t}\left(D^{2}\phi_{t}\right)^{-1}\nabla\xi\right]\,dx	\\
    & =\int_{\partial\Omega}\inner{\left(D^{2}\phi_{t}\right)^{-1}\nabla\xi}{\normal} f_{t}\,dx
    =\int_{\partial\Omega}\left|\left(D^{2}\phi_{t}\right)^{-1}\normal\right|qf_{t}\,dx.
  \end{align*}
  Existence of solutions is then a simple consequence of the Lax--Milgram Theorem and elliptic regularity theory.
\end{proof}
Now we are ready to prove differentiability of $\phi_t$ with respect to $t$.

\begin{proof}[Proof of Theorem \ref{thm:implicit_fun}]
    The previous analysis shows that $D_\phi\Gamma (t, \phi_t)$ defines an isomorphism $\mathcal{X}\longrightarrow\mathcal Y_t$ where
\begin{align*}
  \mathcal{X} &	:=\left\{ \xi\in\mathcal{C}^{2,\beta}(\overline{\Omega})\mid\int_{\Omega}\xi =0\right\}  \\
  \mathcal{Y}_{t} & :=\left\{ (p,q)\in\mathcal{C}^{0,\beta}(\overline{\Omega})\times\mathcal{C}^{1,\beta}(\partial\Omega)\mid\int_{\Omega}pf_{t}=\int_{\partial\Omega}\left|(D^{2}\phi_{t})^{-1}\normal\right|qf_{t}\right\} .
\end{align*}
This is not enough to apply the implicit function theorem directly, but we  are able to apply it to a modified operator. Without loss of generality, let us show that $t\mapsto \phi_t$ is a smooth map in a neighborhood of $t=0$. To do this we define
\begin{equation*}
  \tilde{\Gamma}:[0,1]\times\left(\mathcal{C}_{>0}^{2,\beta}(\overline{\Omega})\cap\mathcal{X}\right)\longrightarrow\mathcal{Y}_{0},
  \quad \tilde{\Gamma}^{(1)}(t, \phi)=\Gamma_{t}^{(1)}(\phi)+C_{\phi},
  \quad \tilde{\Gamma}^{(2)}(t, \phi)=\Gamma_{t}^{(2)}(\phi),
\end{equation*}
where $\tilde{\Gamma}=(\tilde{\Gamma}^{(1)}, \tilde{\Gamma}^{(2)})$, and the constant $C_\phi\in \mathbb R$ is chosen so $\tilde{\Gamma}(t, \phi)\in\mathcal{Y}_{0}$. In other words,
\begin{equation*}
  \tilde{\Gamma}^{(1)}(t,\phi)=
  \Gamma_{t}^{(1)}(\phi)+\int_{\partial\Omega}\left|(D^{2}\phi_{0})^{-1}\normal\right|
    \Gamma_{t}^{(2)}(\phi)f_0-\int_{\Omega}\Gamma_{t}^{(1)}(\phi)f_0.
\end{equation*}
This functional is now different from the original $\Gamma$, so we must check whether it verifies all the hypotheses of the implicit function theorem. Its $\mathcal{C}^m$-smoothness is a simple consequence of the $\mathcal{C}^m$-smoothness of $\Gamma$. Also, given $\xi\in\mathcal{X}$, we have
\begin{equation*}
    D_{\phi}\tilde \Gamma^{(1)}(t, \phi)\xi=D_{\phi}\Gamma_{t}^{(1)}(\phi)\xi+
    \int_{\partial\Omega}\left|(D^{2}\phi_{0})^{-1}\normal\right|\big(D_{\phi}\Gamma_{t}^{(2)}(\phi)\xi\big) f_0
    -\int_{\Omega}\big(D_{\phi}\Gamma_{t}^{(1)}(\phi)\xi\big) f_0.
\end{equation*}
Note that $D_\phi \Gamma (t, \phi_t)\xi\in \mathcal Y_t$, so the last two expressions cancel out if we set $\phi=\phi_0$, $t=0$. Therefore, we have 
\begin{equation*}
  D_\phi\tilde\Gamma(0, \phi_0) = D_\phi\Gamma(0, \phi_0)|_{\mathcal{X}},
\end{equation*}
which shows that $D_\phi\tilde\Gamma(0, \phi_0)$ is an isomorphism and that the implicit function theorem can be applied. This implies the existence of some $\delta>0$ and a smooth mapping
\begin{equation*}
  t\in (-\delta,\delta) \mapsto \tilde \phi_t\in \mathcal C^{2,\beta}_{>0}(\overline \Omega),
\end{equation*}
such that $\tilde\Gamma(t,\tilde \phi_t)\equiv (0, 0)$. As a last step, we show that we actually have $\tilde \phi_t = \phi_t$. We know that for each $t$,
\begin{equation*}
  \det D^{2}\tilde\phi_t(x)=C_t\frac{f_{t}(x)}{g_{t}\big(\nabla\tilde\phi_{t}(x)\big)},\qquad\nabla\tilde\phi_t(\Omega)=\Omega_*.
\end{equation*}
for some constant $C_t>0$. However, using the change of variables $y=\nabla\tilde\phi_t(x)$
\begin{equation*}
  C_t=\int_{\Omega}C_tf_t(x)\,dx=\int_{\Omega}g_t\left(\nabla\tilde\phi_t(x)\right)\det D^{2}\tilde\phi_t(x)\,dx=\int_{\Omega_{*}}g_t(y)\,dy=1,
\end{equation*}
which implies $\tilde \phi_t = \phi_t$ by uniqueness of solutions, finishing the proof of Theorem \ref{thm:implicit_fun}.
\end{proof}

Next, we use an approximation argument to see how Theorem \ref{thm:implicit_fun} may be applied without directly proving that $f_t$, $g_t$ are smooth curves in a Banach space.

\begin{proof}[Proof of Corollary \ref{cor:implicit_fun}]
First let the linear and bounded extension operators
\begin{equation*}
  \mathcal{E}:\mathcal{C}^{0,\gamma}(\overline{\Omega})\longrightarrow\mathcal{C}^{0,\gamma}(\mathbb{R}^{d}),\qquad\mathcal{E}_{*}:\mathcal{C}^{0,\gamma}(\overline{\Omega}_{*})\longrightarrow\mathcal{C}^{0,\gamma}(\mathbb{R}^{d})
\end{equation*}
correspond to $\mathcal{E}_0$ as constructed in \cite[Section 2.2 (8)]{stein1970singular}; these operators are defined independent of $\gamma\in (0, 1]$. Additionally, the extended function can be expressed as a finite sum outside of the original domain, hence differentiation in $t$ commutes with either extension operator. Thus we may assume $f_t$, $g_t$ are defined on $\R^d$ with $f_t$, $g_t\geq \frac{a}{2}$ on a neighborhood of $\overline\Omega$, $\overline\Omega_*$, such that
\begin{align}\label{eqn: extended bounded}
\sup_{t\in (0, 1)}\max\left\{ \Vert f_t\Vert_{\mathcal C^{0,\alpha}(\R^d)},\Vert \partial_tf_t\Vert_{L^\infty(\R^d)},
\Vert g_t\Vert_{\mathcal C^{0,\alpha}(\R^d)}, \Vert \partial_tg_t\Vert_{L^\infty(\R^d)}\right\}<\infty,
\end{align}
and with $\partial_tf_t$, $\partial_tg_t$ are continuous on $(0, 1)\times \R^d$. 
We further extend $f_t$, $g_t$ to $t\in \R$ by setting $f_t=f_0$ for $t<0$ and $f_t=f_1$ for $t>1$, and similarly for $g_t$. Next, let ${\eta\in\mathcal{C}_{c}^{\infty}\left(\mathbb{R}\times\mathbb{R}^{d}\right)}$ be nonnegative, define its rescaling $\eta^{\varepsilon}(t, x)=\varepsilon^{-(d+1)}\eta(t/\varepsilon, x/\varepsilon)$, and let 
\begin{align*}
f_{t}^{\varepsilon}(x):&=\left(\int_\Omega (f*\eta^{\varepsilon})(t, y)dy\right)^{-1}(f*\eta^{\varepsilon})(t, x),\\ 
g_{t}^{\varepsilon}(x):&=\left(\int_{\Omega_*} (g*\eta^{\varepsilon})(t, y)dy\right)^{-1}(g*\eta^{\varepsilon})(t, x),
\end{align*}
so that $f_{t}^{\varepsilon}$, $g_{t}^{\varepsilon}$ are probability densities on $\Omega$, $\Omega_{*}$ for each $t\in [0, 1]$; above the convolutions are taken in both variables $t$ and $x$. Then for each $0<t<1$ we have that $f_{t}^{\varepsilon}$, $g_{t}^{\varepsilon}$ are bounded from below by $\frac{a}{2}$ on $\Omega$ and $\Omega_*$ respectively, and belong to $\mathcal C^{\infty}$ in $(t, x)\in \R\times \R^d$. 

Then, one can see that $t\mapsto f_t^\varepsilon$, $t\mapsto g_t^\varepsilon$ are infinitely Frechét differentiable as curves in any space $\mathcal C^{k,\gamma}$ (see Proposition \ref{prop:Frechet_diff} for the case $k=0$, and note the argument can be extended to any $k\in \mathbb N$.) In particular, we satisfy the hypotheses of Theorem \ref{thm:implicit_fun} with $m=1$. Thus if $\phi_{t}^{\varepsilon}$ is the Brenier potential from $f^\varepsilon_t$ to $g^\varepsilon_t$ with integral zero, it is differentiable with respect to $t$. Fixing $t\in (0, 1)$, by integration by parts with~\eqref{eq:boundary_val_prob_implicit}, $\xi_{t}^{\varepsilon}:=\partial_{t}\phi_{t}^{\varepsilon}$ satisfies for any test function $\theta\in \mathcal  C^\infty(\overline \Omega)$,
\begin{align}
\begin{split}
    \int_\Omega \inner{(D^{2}\phi^\varepsilon_t)^{-1}\nabla\xi^\varepsilon_t}{\nabla \theta}f^\varepsilon_tdx&=
        \int_\Omega\theta\left(\frac{\partial_{t}g^\varepsilon_t\left(\nabla\phi^\varepsilon_t\right)}{g^\varepsilon_t\left(\nabla\phi^\varepsilon_t\right)}-\frac{\partial_{t}f^\varepsilon_t}{f^\varepsilon_t}\right)f^\varepsilon_tdx
\end{split}
 \label{eqn: mollified pde}
\end{align}
Let $\psi^\varepsilon_t$ be the Brenier potential from $g^\varepsilon_t$ to $f^\varepsilon_t$ which is zero at some designated point in $\Omega_*$ for all $\varepsilon$, then by Lemma \ref{lem: global C2 estimate} we obtain 
\begin{align*}
  \left\lVert \phi_{t}^{\varepsilon}\right\rVert _{\mathcal{C}^{2,\alpha}(\overline{\Omega})},\ \left\lVert \psi_{t}^{\varepsilon}\right\rVert _{\mathcal{C}^{2,\alpha}(\overline{\Omega}_*)}\leq C,
\end{align*}
where $C$ depends on $\left\Vert f_{t}^{\varepsilon}\right\Vert _{\mathcal{C}^{0,\alpha}(\overline{\Omega})}$, $\left\Vert g_{t}^{\varepsilon}\right\Vert _{\mathcal{C}^{0,\alpha}(\overline{\Omega}_{*})}$, and the lower bound of $f_{\varepsilon},g_{\varepsilon}$, which by construction are independent of $\varepsilon>0$. Thus we can apply Arzel{\`a}--Ascoli to find subsequences which, for any $0<\beta<\alpha$, converge as $\varepsilon\to 0$ in $\lVert\cdot\rVert_{\mathcal{C}^{2,\beta}(\overline{\Omega})}$ to $\phi_t$ and $\lVert\cdot\rVert_{\mathcal{C}^{2,\beta}(\overline{\Omega}_{*})}$ to $\psi_t$ respectively, which are clearly Brenier potentials from $f_t$ to $g_t$ and $g_t$ to $f_t$ respectively (for example, use that $\phi_t$ and $\psi_t$ are convex as a limit of convex functions, then \cite[Exercise 2.17]{villani2009optimal_old_new} combined with the uniqueness in \cite[Theorem 1.3]{Brenier91}). Since $\nabla \psi^\varepsilon_t=(\nabla \phi^\varepsilon_t)^{-1}$ and $D^2 \psi^\varepsilon_t=(D^2 \phi^\varepsilon_t)^{-1}$, this also yields
\begin{align*}
 \Lambda^{-1}\Id\leq (D^{2}\phi_{t}^{\varepsilon})^{-1}\leq\Lambda\Id
\end{align*}
for $\Lambda>0$ independent of $\varepsilon$. Then by Lemmas \ref{lem:moser_estimate} and \ref{lem:Schauder_estimate} with~\eqref{eqn: mollified pde}, we can find a subsequence such that $\xi^\varepsilon_t$ converges in $\lVert\cdot\rVert_{\mathcal{C}^{1,\beta}(\overline{\Omega})}$ for any $\beta <\alpha$, and the limit can be seen as $\xi_t:=\partial_t\phi_t$ by taking $\varepsilon\to 0$ in the expression
\begin{equation*}
  \int_0^1\int_{\Omega}\xi^{\varepsilon}(t, x)\theta(t,x)\,dx\,dt
  =-\int_0^1\int_{\Omega}\phi^{\varepsilon}(t, x)\partial_{t}\theta(t, x)\,dx\,dt,
\end{equation*}
where $\theta\in\mathcal{C}_{c}^{\infty}\left((0, 1)\times\Omega\right)$ is an arbitrary test function. Thus the left hand side of~\eqref{eqn: mollified pde} converges to the left hand side of~\eqref{eqn: weak PDE nonsmooth} 
as $\varepsilon\to 0$. By the regularity of $\partial_tf_t$, the sequence $\partial_tf^\varepsilon_t$ converges pointwise to $\partial_tf_t$ on $\overline\Omega$ (viewing $\{t\}\times \overline\Omega$ as a compact subset of $(0, 1)\times \R^d$), and similarly for $\partial_tg^\varepsilon_t$. Thus by the bound~\eqref{eqn: extended bounded}, we may apply dominated convergence and let $\varepsilon\to 0$ in~\eqref{eqn: mollified pde}, which yields the weak formulation of~\eqref{eqn: weak PDE nonsmooth} as $\varepsilon\to 0$, finishing the proof.
\end{proof}

\appendix
\section{Regularity of compositions}\label{sec:reg_comp}
All results stated here are taken straight out of \cite{llave1999_reg_comp} with some simplifications. The main goal is to study the regularity of composition of maps in terms of the regularity of each factor.

\textbf{Basic definitions and notation.} We will consider Banach spaces $E$, $F$, $G$ and maps
\begin{equation*}
  f:U\subset E\longrightarrow F,\qquad g:V\subset F\longrightarrow G,
\end{equation*}
where $U$, $V$ are open sets so that $f(U)\subset V$. Hölder spaces will be denoted with only one real index so that $\mathcal C^{k,\alpha}\equiv \mathcal C^{k+\alpha}$ for $k\in \mathbb N$, $0<\alpha <1$. This notation makes it easier to state the required results.

We remind the reader that a map $f:U\subset E\longrightarrow F$ is said to be \textit{Frechét differentiable} at $x\in U$ if there is a bounded linear operator $df(x):E\longrightarrow F$ such that
\begin{equation*}
  \lim_{t\to0}\frac{\left\Vert f(x+th)-f(x)-tdf(x)h\right\Vert_F }{t}=0.
\end{equation*}
Then, by induction, $f$ is said to be $n$ times differentiable if $x\mapsto df(x)$ is $n-1$ times differentiable and we set $d^nf = d(d^{n-1}f)$, where $d^nf(x):E^n\longrightarrow F$ is identified with a multilinear map.

When $n$ is a natural number, we define $\mathcal C^n(U)$ as the set of $n$-times continuously differentiable maps $f:U\subset E\longrightarrow F$ and equip it with the norm
\begin{equation*}
  \left\Vert f\right\Vert _{\mathcal{C}^{n}(U)}=\sum_{k=0}^{n}\sup_{x}\left\Vert d^{k}f(x)\right\Vert_F.
\end{equation*}
When $r>0$ is a nonnegative number of the form $r=n+\alpha$, $n\in \mathbb{N}$, $0<\alpha <1$, we define $\mathcal C^r(U)$ as the set subspace of $\mathcal C^n(U)$ such that the $n$-th derivative is $\alpha$-Hölder continuous. We equip $\mathcal C^r(U)$ with the norm 
\begin{equation*}
  \left\Vert f\right\Vert _{\mathcal{C}^{r}(U)}=\left\Vert f\right\Vert _{\mathcal{C}^{n}(U)}+\sup_{x\neq y}\frac{\left\Vert d^{n}f(x)-d^{n}f(y)\right\Vert_F }{\left\Vert x-y\right\Vert_E^\alpha }.
\end{equation*}

\textbf{Frechét differentiability in relation with regularity in Euclidean space.}
In this work, we consider functions $f:[0,1]\times \overline \Omega\longrightarrow\mathbb R$, with $\Omega\subset \mathbb R^d$ open, and view them as a curve in a Banach space
$$
\mathcal F:[0,1]\longrightarrow \mathcal C^\alpha(\overline \Omega),\qquad\mathcal F(t)(x)=f(t,x).
$$
It is not immediately obvious how classical regularity of $f$ with respect to $(t, x)$ translates into regularity for the map $\mathcal F$. For example, requiring that $f$ be differentiable in $t$, that $\mathcal F(t)\in \mathcal C^\alpha(\overline \Omega)$, and $\mathcal F'(t)\in C^\alpha(\overline \Omega)$ is not sufficient to ensure Frechét differentiability of $\mathcal F$, since this also requires convergence of incremental quotients in $\mathcal C^\alpha(\overline \Omega)$. Nevertheless, we may still conclude that $\mathcal F$ is Frechét differentiable if we slightly strengthen these assumptions.

\begin{proposition}\label{prop:Frechet_diff}
    Consider $f:[0,1]\times\overline{\Omega}\longrightarrow\mathbb{R}$ such that
    \begin{enumerate}
        \item There is $0<\alpha<1$ with $f(t,\cdot)\in\mathcal{C}^{0, \alpha}(\overline{\Omega})$ for each $t$.
        \item $f$ is differentiable in $t$ over $(0,1)\times\overline \Omega$, and $\partial_{t}f$ is continuous over $[0,1]\times\overline{\Omega}$. 
        \item For each $t$, $\partial_{t}f(t,\cdot)\in\mathcal{C}^{0, \alpha}(\overline{\Omega})$ with $M:=\sup_{t\in(0,1)}\Vert\partial_{t}f(t,\cdot)\Vert_{\mathcal{C}^{0, \alpha}(\overline{\Omega})}<\infty$ .
    \end{enumerate}
    Then, for any $0<\beta<\alpha$, the mapping 
    $$
    \mathcal F:[0,1]\longrightarrow \mathcal C^{0, \beta}(\overline \Omega),\qquad\mathcal F(t)(x):=f(t,x),
    $$
    is Frechét differentiable with continuous differential.
\end{proposition}

\begin{proof}
    The clear candidate for differential $d\mathcal{F}(t):\mathbb{R}\longrightarrow\mathcal{C}^{0,\beta}(\overline{\Omega})$ is $\left\langle d\mathcal{F}(t),h\right\rangle =\partial_{t}f(t,\cdot)h$, which we can identify with $\mathcal{F}'(t):=\partial_{t}f(t,\cdot)\in\mathcal{C}^{0,\beta}(\overline{\Omega})$. We must show that for $t\in (0,1)$, 
    \begin{equation}\label{eq:defn_Frechet_diff}
        \lim_{h\to0}\left\Vert \frac{\mathcal{F}(t+h)-\mathcal{F}(t)}{h}-\mathcal{F}'(t)\right\Vert _{\mathcal{C}^{0, \beta}(\overline{\Omega})}=0.
    \end{equation}
    Using the inequality $\Vert g\Vert_{\mathcal{C}^{0, \beta}(\overline{\Omega})}\leq C\Vert g\Vert_{\mathcal{C}^{0}(\overline{\Omega})}^{1-\frac{\beta}{\alpha}}\Vert g\Vert_{\mathcal{C}^{0, \alpha}(\overline{\Omega})}^{\frac{\beta}{\alpha}}$, we obtain when $0<t+h<1$,
    \begin{align*}
        \left\Vert \frac{\mathcal{F}(t+h)-\mathcal{F}(t)}{h}-\mathcal{F}'(t)\right\Vert _{\mathcal{C}^{0, \beta}(\overline{\Omega})} &
        =\left\Vert \int_{0}^{1}\Big(\partial_{t}f(t+sh,\cdot)-\partial_{t}f(t,\cdot)\Big)\,ds\right\Vert _{\mathcal{C}^{0, \beta}(\overline{\Omega})}\\
        & \leq CM^{\frac{\beta}{\alpha}}\int_{0}^{1}\left\Vert \partial_{t}f(t+sh,\cdot)-\partial_{t}f(t,\cdot)\right\Vert _{\mathcal{C}^{0}(\overline{\Omega})}^{1-\frac{\beta}{\alpha}}\,ds.
    \end{align*}
    Note that ${\displaystyle \lim_{h\to0}}\left\Vert \partial_{t}f(t+sh,\cdot)-\partial_{t}f(t,\cdot)\right\Vert _{\mathcal{C}^{0}(\overline{\Omega})}=0$ for each $s$ due to the uniform continuity of $\partial_{t}f$, so the dominated convergence theorem shows \eqref{eq:defn_Frechet_diff} as desired. The continuity of the differential follows from a similar calculation, as
    $$
    \left\Vert \mathcal{F}'(t+h)-\mathcal{F}'(t)\right\Vert _{\mathcal{C}^{0, \beta}(\overline{\Omega})}\leq
    CM^{\frac{\beta}{\alpha}}\left\Vert \partial_{t}f(t+h,\cdot)-\partial_{t}f(t,\cdot)\right\Vert _{\mathcal{C}^{0}(\overline{\Omega})}^{1-\frac{\beta}{\alpha}}
    \overset{h\to0}{\longrightarrow}0.\qedhere
    $$
\end{proof}
In particular, if $f$ is $\mathcal C^2$ up to the boundary with respect to $(t, x)$, then $\mathcal F$ is continuously differentiable for any $0<\beta<1$. There are straightforward generalizations of the previous Proposition to higher regularity settings.

\textbf{Results on the regularity of composition.}
Next, we state the results required for the work, beginning with Theorem 4.3 of \cite{llave1999_reg_comp}.

\begin{lemma}\label{lem:reg_of_comp1}
  Assume $U\subset E$, $V\subset F$ are convex and let $r=n+\alpha$, $s=m+\beta$  where $0\leq\alpha$, $\beta<1$, $n$, $m\in\mathbb{N}$ and $f\in\mathcal{C}^{r}(U,F)$, $g\in\mathcal{C}^{s}(V,G)$ such that $f(U)\subset V$. Then $g\circ f\in\mathcal{C}^{t}(U,G)$ where $t$ can be taken as:
  \begin{enumerate}
    \item If $n=m=0$ then $t=rs$ and 
    \begin{equation*}
      \left\Vert g\circ f\right\Vert _{\mathcal{C}^{rs}}\le\left\Vert g\right\Vert _{\mathcal{C}^{s}}\left\Vert f\right\Vert _{\mathcal{C}^{r}}^{s}+\left\Vert g\right\Vert _{\mathcal{C}^{0}}.
    \end{equation*}
    \item If $n\geq1$ or $m\geq1$ then $t=\min(r,s)$. Moreover
    \begin{enumerate}
      \item If $n\geq1$ and $0<s<1$ then
      \begin{equation*}
        \left\Vert g\circ f\right\Vert _{\mathcal{C}^{s}}\le C\left\Vert g\right\Vert _{\mathcal{C}^{s}}\left\Vert f\right\Vert _{\mathcal{C}^{1}}^{s}+\left\Vert g\right\Vert _{\mathcal{C}^{0}}.
      \end{equation*}
      \item If $m\geq1$ and $0<r<1$ then
      \begin{equation*}
        \left\Vert g\circ f\right\Vert _{\mathcal{C}^{r}}\le C\left\Vert g\right\Vert _{\mathcal{C}^{1}}\left\Vert f\right\Vert _{\mathcal{C}^{r}}+\left\Vert g\right\Vert _{\mathcal{C}^{0}}.
      \end{equation*}
      \item If $r\geq1$, $s\geq1$ then
      \begin{equation*}
        \left\Vert g\circ f\right\Vert _{\mathcal{C}^{t}}\le C\left\Vert g\right\Vert _{\mathcal{C}^{t}}\left(1+\left\Vert f\right\Vert _{\mathcal{C}^{t}}^{t}\right).
      \end{equation*}
    \end{enumerate}
  \end{enumerate}
\end{lemma}

Now we collect a simplified version of Proposition 6.7 and Theorem 6.10 from \cite{llave1999_reg_comp} which is adapted for our purposes.

\begin{lemma}\label{lem:reg_of_comp2}
  Let $m\in \mathbb{N}$, and $r,s,t\geq 0$ satisfying
  \begin{equation*}
    t\leq
    \begin{cases}
      rs, & \textnormal{if }0\leq r,s<1,\\
      \min(r,s), & \textnormal{otherwise}.
    \end{cases}
  \end{equation*}
  Further assume one of the following:
  \begin{enumerate}[label=(\alph*)]
    \item \label{item:reg_comp2_p1} If $0<t<1$, then $t<s-m$, $t\leq r$, $t<r(s-m)$.
    \item \label{item:reg_comp2_p2} If $t=k+\gamma$ where $k\geq1$, $0\leq\gamma<1$, then $t<s-m$, $t\leq r$.
  \end{enumerate}
  Also consider convex open sets $U\subset E$, $V\subset F$, and $\mathcal U\subset \mathcal C^r(U, F)$ open such that for every $f\in \mathcal U$ we have $\textnormal{dist}(f(U),V^c)>0$. Then,
  \begin{enumerate}
      \item \label{item:g*_comp} Given $g\in\mathcal{C}^{s}(V,G)$, the mapping $g_{*}:\mathcal{U}\subset\mathcal{C}^{r}(U,F)\longrightarrow\mathcal{C}^{t}(U,G)$ given by $g_{*}(f)=g\circ f$ is well-defined and $m$ times continuously differentiable.
      \item \label{item:comp_reg} The composition map $\textnormal{Comp}:\left[\mathcal{U}\subset\mathcal{C}^{r}(U,F)\right]\times\mathcal{C}^{s}(V,G)\longrightarrow\mathcal{C}^{t}(U,G)$ given by $\textnormal{Comp}(f,g)=g\circ f$ is also $m$ times continuously differentiable.
  \end{enumerate}
\end{lemma}
We note that in Section 6 of \cite{llave1999_reg_comp} one finds an additional geometrical hypothesis h.4 which is not included in the previous statement. This is due to the fact that h.4 is automatically satisfied when $V$ is convex, since one has that $V_\varepsilon:=\{y\in V:\textnormal{dist}(y,V^c)>\varepsilon\}$ is also convex.

Finally, we will also make use of the following result.

\begin{lemma}\label{lem:prod_quotient_Holder}
  Consider a bounded domain $\Omega\subset \mathbb{R}^n$ and $f\in \mathcal C^s(\Omega)$, $g\in \mathcal C^s(\Omega)$. Then there is a constant $C=C(s,\textnormal{diam}(\Omega))$ such that
  \begin{equation}\label{eq:prod_holder}
    \left\Vert f\cdot g\right\Vert _{\mathcal{C}^{s}(\Omega)}\leq 
    C\left\Vert f\right\Vert _{\mathcal{C}^{s}(\Omega)}\left\Vert g\right\Vert _{\mathcal{C}^{s}(\Omega)}.
  \end{equation}
  If we further assume that $\inf_\Omega g>0$ then there is $C=C(s,\inf g,\textnormal{diam}(\Omega))$ such that
  \begin{equation}\label{eq:quot_holder}
    \left\Vert f/g\right\Vert _{\mathcal{C}^{s}(\Omega)}\leq 
    C\left\Vert f\right\Vert _{\mathcal{C}^{s}(\Omega)}(1+\left\Vert g\right\Vert _{\mathcal{C}^{s}(\Omega)}).
  \end{equation}
\end{lemma}
The estimate (\ref{eq:prod_holder}) can be found in \cite[Sec. 4.1]{gilbarg_trudinger1977elliptic}, while (\ref{eq:quot_holder}) is immediately derived from Lemma \ref{lem:reg_of_comp1} and (\ref{eq:prod_holder}).

\section{A quantitative $\mathcal C^{2,\alpha}$ estimate for the Brenier potential}\label{sec:compactness_ellipticity}
In this appendix we prove Lemma~\ref{lem: global C2 estimate}. Our argument only relies on the $\mathcal C^{1,\alpha}$ estimate from \cite{caffarelli1992boundary_reg_convex_potential} and the fact that the Brenier potential is $\mathcal C^{2,\alpha}$ when the densities are $\mathcal C^{0,\alpha}$, which was proven in \cite{chenLiuWang2021MAglobalReg}. Let us begin by motivating the main idea behind the proof.

Consider $f\in \mathcal C^{0,\alpha}(\overline\Omega)$, $g\in \mathcal C^{0,\alpha}(\overline \Omega_*)$, let $\phi$ be the solution of \eqref{eq:MA_appendix}. Assume we are able to prove that 
\begin{equation*}
  (f,g)\in\mathcal{C}^{0,\alpha}(\overline{\Omega})\times\mathcal{C}^{0,\alpha}(\overline{\Omega}_{*})
  \mapsto\phi\in\mathcal{C}^{2,\alpha}(\overline{\Omega})
\end{equation*}
  is continuous as a map between Banach spaces. Then the conclusion of Lemma~\ref{lem: global C2 estimate} would immediately follow, as a result of the compactness of the following set in $\mathcal{C}^{0,\alpha}(\overline{\Omega})\times\mathcal{C}^{0,\alpha}(\overline{\Omega}_{*})$:
\begin{equation*}
  \left\{ (f,g)\in\mathcal{C}^{0,\alpha}(\overline{\Omega})\times\mathcal{C}^{0,\alpha}(\overline{\Omega}_{*})\mid f,g\geq a,\;\left\Vert f\right\Vert _{\mathcal{C}^{0,\beta}(\overline{\Omega})},\left\Vert g\right\Vert _{\mathcal{C}^{0,\beta}(\overline{\Omega}_{*})}\leq M\right\} .
\end{equation*}
Then, one may attempt to show the continuity (smoothness) of $(f,g)\mapsto\phi$ through the implicit function theorem. This would involve the definition of a functional similar to one in  Section \ref{sec:time_reg_ifc}, which would be of the form
\begin{align*}
  & \Gamma:\mathcal{C}^{0,\alpha}(\overline{\Omega})\times\mathcal{C}^{0,\alpha}(\overline{\Omega}_{*})\times\mathcal{C}^{2,\alpha}(\overline{\Omega})\longrightarrow\mathcal{C}^{0,\alpha}(\overline{\Omega})\times\mathcal{C}^{1,\alpha}(\partial\Omega)	\\
  & \Gamma(f,g,\phi):=\left(\log\det D^{2}\phi+\log g(\nabla\phi)-\log f,\Gamma^{(2)}(f,g,\phi)\right),
\end{align*}
where the coordinate $\Gamma^{(2)}$ is chosen to enforce the boundary condition $\nabla\phi(\Omega)=\Omega_{*}$. However this runs into a technical problem: in order to apply the implicit function theorem, we need $\Gamma$ to be at least $\mathcal{C}^{1}$ as a map between Hölder spaces. Unfortunately, for $(g,\phi)\mapsto g(\nabla\phi)$ to be a smooth map, $g$ needs to belong to a regularity class better than $\mathcal C^1$ (see Appendix \ref{sec:reg_comp} or \cite{llave1999_reg_comp} for further details on the regularity of composition). 

In this manuscript we avoid this problem by first considering the case where $g$ is constant. To be precise, we consider the mapping
\begin{equation*}
  f\in\mathcal{C}_{>0}^{0,\alpha}(\overline{\Omega})\mapsto\phi_{f}\in\mathcal{C}^{2,\alpha}(\overline{\Omega}),
\end{equation*}
where $\mathcal{C}_{>0}^{0,\alpha}(\overline{\Omega})$ denotes the subset of $\mathcal{C}^{0,\alpha}(\overline{\Omega})$ formed by functions bounded away from $0$ and $\phi_{f}$ is the unique convex function with average $0$ solving
\begin{equation}\label{eq:MA_g_constant}
  \begin{cases}
  \det D^{2}\phi_{f}=\left\langle f\right\rangle _{\Omega}^{-1}|\Omega_{*}|f, & \textnormal{in }\Omega,\\
  \nabla\phi_{f}(\Omega)=\Omega_{*},
  \end{cases}
\end{equation}
where $\left\langle f\right\rangle _{\Omega}:=\int_{\Omega}f$. %
Note that the factor $\left\langle f\right\rangle _{\Omega}^{-1}|\Omega_{*}|$ only appears so that the mass balance condition is satisfied for the optimal transport problem. Since $\partial\Omega$ and $\partial\Omega_*$ are $\mathcal{C}^{1, 1}$ regular, by \cite[Theorem 1.1]{chenLiuWang2021MAglobalReg}, for any $f\in \mathcal{C}^{0, \alpha}_{>0}(\overline\Omega)$, the above solution $\phi_f$ belongs to $\mathcal{C}^{2, \alpha}(\overline\Omega)$, which is uniformly convex by Remark~\ref{rmk:ellipticity_D2phi}; in particular this shows the map $f\mapsto \phi_f$ is well-defined on $\mathcal{C}^{0, \alpha}_{>0}(\overline\Omega)$. In this setting, we are able to prove the following lemma:
\begin{lemma} \label{lem:workaround}
  Assume $\Omega$, $\Omega_{*}\subset\mathbb{R}^{d}$ are convex sets where $\partial\Omega,\,\partial\Omega_*\in \mathcal{C}^{2,\beta}$. Then for any $\alpha<\beta$, $f\in\mathcal{C}_{>0}^{0,\alpha}(\overline{\Omega})\mapsto\phi_{f}\in\mathcal{C}^{2,\alpha}(\overline{\Omega})$ is of class $\mathcal C^1$ when viewed as a mapping between Banach spaces.
\end{lemma}
We will also make use of the following estimate taken from  \cite{caffarelli1992boundary_reg_convex_potential}:
\begin{lemma}[Caffarelli] \label{lem:Caff_C1_estimate}
  Let $\Omega,\Omega_{*}\subset\mathbb{R}^{d}$ be bounded convex sets and consider two probability densities $f$, $g$ over $\Omega$, $\Omega_{*}$ such that $a\le f,g\leq A$ for some constants $a,A>0$. Let $\phi$  be the unique convex potential with average $0$ inducing the optimal map $(\nabla\phi)_{\sharp}f=g$. Then there exists $\delta>0$ such that
  \begin{equation*}
    \left\Vert \phi\right\Vert _{\mathcal{C}^{1,\delta}(\overline{\Omega})}\leq C,
  \end{equation*}
  where both $\delta$ and $C$ depend on $a$, $A$ and the maximum and minimum diameter of $\Omega$, $\Omega_{*}$.
\end{lemma}
Before proving Lemma \ref{lem:workaround}, let us show that it implies Lemma~\ref{lem: global C2 estimate}.

\begin{proof}[Proof of Lemma~\ref{lem: global C2 estimate}]
Let $\phi$ be the solution of \eqref{eq:MA_appendix} and fix $\alpha<\beta$. By Lemma \ref{lem:Caff_C1_estimate}, there exists $\delta>0$ such that $\phi \in \mathcal C^{1,\delta}$ with an estimate on $\Vert \phi\Vert_{\mathcal C^{1,\delta}(\overline \Omega)}$. Therefore, letting $\tilde{\alpha}:=\min(\alpha, \delta\beta)<\beta$, Lemmas \ref{lem:reg_of_comp1} and \ref{lem:prod_quotient_Holder} show that 
  \begin{equation*}
    \left\Vert f/g(\nabla\phi)\right\Vert _{\mathcal{C}^{0,\tilde{\alpha}}(\overline{\Omega})}
    \leq C\left(d,a,\Omega,\Omega_{*},\left\Vert f\right\Vert _{\mathcal{C}^{0,\beta}(\overline{\Omega})},\left\Vert g\right\Vert _{\mathcal{C}^{0,\beta}(\overline{\Omega}_*)}\right)=:M.
  \end{equation*}
Now note that $\phi=\phi_{f/g(\nabla\phi)}\in C^{2, \alpha}(\overline\Omega)$ in the notation of Lemma~\ref{lem:workaround}. Since the set 
\begin{equation*}
  \left\{ h\in\mathcal{C}^{0,\tilde{\alpha}}(\overline{\Omega})\mid h\geq a/\lVert g\rVert_{\mathcal{C}^{0}(\overline{\Omega}_*)},\;\left\Vert h\right\Vert _{\mathcal{C}^{0,\tilde{\alpha}}(\overline{\Omega})}\leq M\right\}
\end{equation*}
is compact as a subset of $\mathcal{C}^{0, \tilde{\alpha}-\varepsilon}(\overline\Omega)$, for $0<\varepsilon<\tilde \alpha$ arbitrary, applying Lemma~\ref{lem:workaround} shows 
  \begin{equation}\label{eq:workaround_estimate_only_f}
    \left\Vert \phi\right\Vert _{\mathcal{C}^{2,\tilde{\alpha}-\varepsilon}(\Omega)}
    =\left\Vert \phi_{f/g(\nabla\phi)}\right\Vert _{\mathcal{C}^{2,\tilde{\alpha}-\varepsilon}(\Omega)}
    \leq C\left(d,a,\Omega,\Omega_{*},\left\Vert f\right\Vert _{\mathcal{C}^{0,\beta}(\overline{\Omega})}, \left\Vert g\right\Vert _{\mathcal{C}^{0,\beta}(\overline{\Omega}_*)},\alpha,\beta\right).
  \end{equation}
  In particular this implies a $\mathcal C^2$ bound on $\phi$, thus we can bootstrap and use Lemmas \ref{lem:reg_of_comp1}, \ref{lem:prod_quotient_Holder} again to conclude
  \begin{equation*}
    \left\Vert f/g(\nabla\phi)\right\Vert _{\mathcal{C}^{0,\beta}(\overline{\Omega})}\leq C\left(d,a,\Omega,\Omega_{*},\left\Vert f\right\Vert _{\mathcal{C}^{0,\beta}(\overline{\Omega})},\left\Vert g\right\Vert _{\mathcal{C}^{0,\beta}(\overline{\Omega}_*)},\alpha,\beta\right),
  \end{equation*}
  which can be combined with Lemma~\ref{lem:workaround} again to produce the desired estimate.
\end{proof}
Now it only remains to prove Lemma \ref{lem:workaround}, which follows the same outline used in Section \ref{sec:time_reg_ifc} and in \cite{gonzalez2024_linearization_MA}. 

\begin{proof}[Proof of Lemma \ref{lem:workaround}]
The idea is to apply the implicit function theorem to the functional 
\begin{equation*}
  \Gamma:\mathcal{C}_{>0}^{0,\alpha}(\overline{\Omega})\times\mathcal{C}_{>0}^{2,\alpha}(\overline{\Omega})\longrightarrow\mathcal{C}^{0,\alpha}(\Omega),\qquad\Gamma(f,\phi):=\left(\det D^{2}\phi-|\Omega_{*}|\left\langle f\right\rangle _{\Omega}^{-1}f,\omega_{*}(\nabla\phi)\right),
\end{equation*}
where recall that $\omega_{*}$ is a convex defining function of $\Omega_{*}$. The domain of $\Gamma$ is defined by the  open sets
\begin{align*}
    \mathcal{C}_{>0}^{0,\alpha}(\overline{\Omega}) & 
    :=\left\{ f\in\mathcal{C}^{0,\alpha}(\overline{\Omega})\mid\inf_{\Omega}f>0\right\}, \\
    \mathcal{C}_{u}^{2,\alpha}(\overline{\Omega}) &
    :=\left\{ \phi\in\mathcal{C}^{2,\alpha}(\overline{\Omega})\mid\inf_{x\in\Omega,|e|=1}\inner{D^{2}\phi(x)e}{e} >0\right\} .
\end{align*}
From Lemma \ref{lem:reg_of_comp2} we immediately see that $\Gamma$ is of class $\mathcal C^1$ with differential
\begin{equation*}
  D_{\phi}\Gamma(f,\phi)\xi=\left(\textnormal{tr}\left[A[\phi]D^{2}\xi\right],\inner{\normal_{*}\left(\nabla\phi\right)}{\nabla\xi}\right),
  \qquad A[\phi] = (\det D^2\phi) \left(D^2\phi\right)^{-1},
\end{equation*}
where $A[\phi]$ is the cofactor matrix of $D^2\phi$. Therefore, showing that $D_\phi \Gamma (f,\phi)$ is a diffeomorphism is equivalent to showing that for each  $p\in\mathcal{C}^{0,\alpha}(\overline{\Omega})$, $q\in\mathcal{C}^{1,\alpha}(\partial\Omega)$, there is a unique solution $\xi \in \mathcal C^{2,\alpha}(\overline \Omega)$ of
\begin{equation}\label{eq:PDE1_workaround}
  \begin{cases}
    \textnormal{tr}\left[A[\phi]D^{2}\xi\right]=p, & \textnormal{in }\Omega\\
    \inner{\normal_{*}(\nabla\phi)}{\nabla\xi}=q, & \textnormal{on }\partial\Omega.
  \end{cases}
\end{equation}
To study this problem, we will re-write its boundary condition and note that it is equivalent to the same problem written in divergence form. This would be immediately obvious if $\phi$  were $\mathcal{C}^{3}$, since it is known that cofactor matrices are divergence free.

\begin{claim*}
    Let $\phi = \phi_f$ be the solution of \eqref{eq:MA_g_constant}. A function $\xi\in\mathcal{C}^{2,\alpha}(\overline{\Omega})$ solves \eqref{eq:PDE1_workaround} if and only if it is a weak solution of
  \begin{equation}\label{eq:PDE2_workaround}
    \begin{cases}
      \div\left(A[\phi]\nabla\xi\right)=p, & \textnormal{in }\Omega\\
      \inner{A[\phi]\nabla\xi}{\normal} =\left|A[\phi]\normal\right|q, & \textnormal{on }\partial\Omega.
    \end{cases}
  \end{equation}
  As a consequence, \eqref{eq:PDE1_workaround} has a solution if and only if the following compatibility condition holds:
  \begin{equation*}
    \int_{\Omega}p = \int_{\partial\Omega}\left|A[\phi]\normal\right|q.
  \end{equation*}
  Moreover, such a solution is unique up to additive constants.
\end{claim*}
\begin{proof}[Proof of the claim]
    Noting that $D^2\phi>0$, the boundary condition in~\eqref{eq:PDE2_workaround} is equivalent to $\inner{(D^2\phi)^{-1}\nabla\xi}{\normal} =\left|(D^2\phi)^{-1}\normal\right|q$, hence is equivalent to the boundary condition in~\eqref{eq:PDE1_workaround} as in the proof of Lemma \ref{lem:magic_lemma}. 
    Assume we have a solution $\xi\in \mathcal{C}^{2, \alpha}(\overline\Omega)$ of \eqref{eq:PDE1_workaround} and let $\phi_{\varepsilon}$ be a mollification of $\phi$. Then, using that $\div(A[\phi_{\varepsilon}])=0$, we have
  \begin{equation*}
    \int_{\Omega}\textnormal{tr}\left(A[\phi_{\varepsilon}]D^{2}\xi\right)\theta=\int_{\Omega}\div\left(A[\phi_{\varepsilon}]\nabla\xi\right)\theta=-\int_{\Omega}\inner{A[\phi_{\varepsilon}]\nabla\xi}{\nabla\theta}+\int_{\partial\Omega}\inner{A[\phi_{\varepsilon}]\nabla\xi}{\normal} \theta.
  \end{equation*}
  Letting $\varepsilon\to0$ we see that $\xi$  satisfies the weak formulation of (\ref{eq:PDE2_workaround}). Now assume $\xi$  is a weak solution of (\ref{eq:PDE2_workaround}) and consider $\xi_{\varepsilon}$ where for each $\varepsilon>0$ we have
  \begin{equation*}
    \begin{cases}
      \div\left(A[\phi_{\varepsilon}]\nabla\xi_\varepsilon\right)=p+c_\varepsilon, & \textnormal{in }\Omega\\
      \left\langle A[\phi_{\varepsilon}]\nabla\xi_\varepsilon,\normal\right\rangle =\left|A[\phi_\varepsilon]\normal\right|q, & \textnormal{on }\partial\Omega,
    \end{cases}
  \end{equation*}
  where $c_\varepsilon\in \mathbb R$ is a constant with $c_\varepsilon\overset{\varepsilon\to0}{\longrightarrow}0$ ensuring the compatibility condition holds.
  Using that the previous equation is also written in non-divergence form, the classical Schauder estimates (see \cite[Thm. 6.30]{gilbarg_trudinger1977elliptic}) tell us that $\xi_{\varepsilon}$ has a bound $\left\Vert \xi_{\varepsilon}\right\Vert _{\mathcal{C}^{2,\alpha}(\overline{\Omega})}\leq C$ independent of $\varepsilon$. Then, a simple compactness argument tells us that $\xi_{\varepsilon}\to\xi$  in $\mathcal{C}^{2}$ as $\varepsilon\to0$. Then, we can let $\varepsilon\to 0$ in the equalities
  \begin{equation*}
    \textnormal{tr}\left(A[\phi_{\varepsilon}]D^{2}\xi_{\varepsilon}\right)=p\textnormal{ in }\Omega,\qquad\qquad\inner{A[\phi_{\varepsilon}]\nabla\xi_{\varepsilon}}{\normal} =\left|A[\phi_{\varepsilon}]\normal\right|q\textnormal{ on }\partial\Omega,
  \end{equation*}
  which hold pointwise. Finally, since $\partial\Omega$ is $\mathcal{C}^{2, \beta}$ regular, the existence of globally regular solutions under the compatibility condition is a simple consequence of the Lax--Milgram Theorem and elliptic regularity.
\end{proof}

Next, we consider a fixed $f_0\in \mathcal C_{>0}^{0,\alpha}(\overline \Omega)$, its corresponding potential $\phi_{f_0}\in \mathcal C^{2,\alpha}(\overline \Omega)$ and will show that the mapping $f\mapsto\phi_{f}$ is $\mathcal{C}^{1}$ in a neighborhood of $f_0$. To do this, first note the previous claim shows that $D_{\phi}\Gamma(f_0,\phi_{f_0})$ is an isomorphism from $\mathcal{X}$ to $\mathcal{Y}_{\phi_0}$ defined by 
\begin{align*}
    \mathcal{X} & :=\left\{ \xi\in\mathcal{C}^{2,\alpha}(\overline{\Omega})\mid\int_{\Omega}\xi=0\right\},\\
    \quad\mathcal{Y}_{\phi_{0}}& := \left\{ (p,q)\in\mathcal{C}^{0,\alpha}(\overline{\Omega})\times\mathcal{C}^{1,\alpha}(\partial\Omega)\mid\int_{\Omega}p=\int_{\partial\Omega}\left|A[\phi_{f_0}]\normal\right|q\right\} .
\end{align*}
Then define, writing $\Gamma(f, \phi)=(\Gamma^{(1)}(f,\phi), \Gamma^{(2)}(f,\phi))$,
\begin{equation*}
  \tilde{\Gamma}:\mathcal{C}_{>0}^{0,\alpha}(\overline{\Omega})\times\left(\mathcal{X}\cap\mathcal{C}_{u}^{2,\alpha}(\overline{\Omega})\right)\longrightarrow\mathcal{Y}_{\phi_{0}},\qquad
  \tilde{\Gamma}(f,\phi)=\left(\Gamma^{(1)}(f,\phi) - C_{f,\phi}\left\langle f\right\rangle _{\Omega}^{-1}f,\Gamma^{(2)}(f,\phi)\right),
\end{equation*}
where $C_{f,\phi}$ is a constant chosen to ensure $\tilde{\Gamma}(f,\phi)\in\mathcal{Y}_{\phi_{0}}$, i.e.
\begin{equation*}
  C_{f,\phi}=\int_{\Omega}\Gamma^{(1)}(f,\phi)-\int_{\partial\Omega}\left|A[\phi_{0}]\normal\right|\Gamma^{(2)}(f,\phi).
\end{equation*}
We now check that $\tilde{\Gamma}$ satisfies the hypotheses of the implicit function theorem. The smoothness of $\tilde{\Gamma}$ is a consequence of the smoothness of $\Gamma$, and we calculate
\begin{equation*}
  D_{\phi}\tilde{\Gamma}(f,\phi)\xi=D_{\phi}\Gamma(f,\phi)\xi-
  \left(\int_{\Omega}D_{\phi}\Gamma^{(1)}(f,\phi)\xi-\int_{\partial\Omega}\left|A[\phi_{0}]\normal\right|D_{\phi}\Gamma^{(2)}(f,\phi)\xi,0\right)\left\langle f\right\rangle _{\Omega}^{-1}f.
\end{equation*}
Note that the second term is $0$ at $(f_0,\phi_{0})$ because $D_{\phi}\Gamma(f_0,\phi_{0})\xi\in\mathcal{Y}_{\phi_{0}}$, hence 
\begin{equation*}
  D_{\phi}\tilde{\Gamma}(f_0,\phi_{0})=D_{\phi}\Gamma(f_0,\phi_0)|_{\mathcal{X}},
\end{equation*}
showing that $D_{\phi}\tilde{\Gamma}(f_0,\phi_{0}):\mathcal{X}\longrightarrow\mathcal{Y}_{\phi_{0}}$ is a diffeomorphism and the implicit function theorem applies. Thus there is a neighborhood $\mathcal{U}\subset\mathcal{C}_{>0}^{0,\alpha}(\Omega)$ of $f_0$ and a $\mathcal C^1$ mapping 
\begin{equation*}
  f\in\mathcal{U}\mapsto\tilde{\phi}_{f}\in\mathcal{C}_{u}^{2,\alpha}(\Omega)\cap\mathcal{X},\quad\textnormal{such that }\tilde{\Gamma}(f,\tilde{\phi}_{f})=0.
\end{equation*}
Finally, we verify that $\tilde \phi_f = \phi_f$. From $\tilde{\Gamma}(f,\tilde{\phi}_{f})=0$ we conclude that for some $C\in \R$,
\begin{equation*}
  \begin{cases}
  \det D^{2}\tilde{\phi}_{f}=\left\langle f\right\rangle _{\Omega}^{-1}\left(|\Omega_{*}|-C\right)f, & \textnormal{in }\Omega\\
  \nabla\tilde{\phi}_{f}(\Omega)=\Omega_{*}.
  \end{cases}
\end{equation*}
Since $\tilde{\phi}_{f}$ is strictly convex, $\nabla\tilde{\phi}_f$ is a one-to-one mapping and we can perform the change of variables $y=\nabla\tilde\phi_{f}(x)$, which shows
\begin{equation*}
  |\Omega_{*}|=\int_{\nabla\tilde{\phi}_{f}(\Omega)}1\,dy=\int_{\Omega}\det D^{2}\tilde{\phi}_{f}(x)\,dx=\left\langle f\right\rangle _{\Omega}^{-1}\left(|\Omega_{*}|-C\right)\int_{\Omega}f=|\Omega_{*}|-C.
\end{equation*}
Therefore, we have $C=0$ and $\tilde \phi_f = \phi_f$, finishing the proof.
\end{proof}

\section*{Acknowledgments}
The authors acknowledge Maria Gualdani for constant support and insightful discussions.
MGD is partially supported by NSF Grant DMS-2205937. JK is partially supported by NSF Grant DMS-2246606. 

\bibliographystyle{unsrt}  
\bibliography{references}
\end{document}